\documentclass[12pt]{amsart}
\usepackage{amssymb}
\usepackage[all]{xy}
\usepackage[utf8]{inputenc}
\usepackage{indentfirst}
\usepackage{amsfonts}
\usepackage{times}
\usepackage{cancel}
\usepackage{bm}
\usepackage{bbding}
\usepackage{mathptmx}
\usepackage{amsmath}
\usepackage{setspace}
\usepackage[usenames]{color}
\usepackage{mathrsfs}
\usepackage{stmaryrd}
\usepackage{amsfonts}
\usepackage{amscd}
\usepackage{cite}
\usepackage{cases}
\usepackage{amsthm}
\usepackage[all]{xy}           %xypic macro for latex2.09
\usepackage{bbm}
\usepackage{txfonts}
\usepackage{amscd}
\usepackage{tikz}
\usetikzlibrary{matrix}
\usepackage{xcolor}
\usepackage[shortlabels]{enumitem}
\usepackage{amsmath} % 必选，提供\mathcal

\usepackage{ifpdf}
\ifpdf
\usepackage[colorlinks=true,linkcolor=blue,filecolor=blue,citecolor=blue,final,backref=page,hyperindex]{hyperref}
\else
\usepackage[colorlinks=true,linkcolor=blue,filecolor=blue,citecolor=blue,final,backref=page,hyperindex]{hyperref}
\fi
\usepackage{tikz}
\usepackage[active]{srcltx}
\usetikzlibrary{positioning, arrows.meta}

\makeatletter
\newcommand {\emptycomment}[1]{}

\numberwithin{equation}{section}
\newtheorem{theorem}{Theorem}[section]
\newtheorem{proposition}[theorem]{Proposition}
\newtheorem{lemma}[theorem]{Lemma}

\theoremstyle{definition}
\newtheorem{remark}[theorem]{Remark}
\newtheorem{definition}[theorem]{Definition}

\numberwithin{equation}{section}

\allowdisplaybreaks[4]
\begin{document}
		\title[]{$\omega$-Lie bialgebras and $\omega$-Yang-Baxter equation}
	
	\author{Yining Sun}
	\address{Y. Sun: School of
		Mathematics and Statistics, Northeast Normal
		University, Changchun 130024, China}
	\email{yiningsun@nenu.edu.cn}
	
	\author{Zeyu Hao}
	\address{Z. Hao: School of
		Mathematics and Statistics, Northeast Normal
		University, Changchun 130024, China}
	\email{haozeyu@nenu.edu.cn}
	
	\author{Ziyi Zhang}
	\address{Z. Zhang: School of
		Mathematics and Statistics, Northeast Normal
		University, Changchun 130024, China}
	\email{zhangziyi@nenu.edu.cn}
	
	\author{Liangyun Chen$^{*}$}
	\address{L. Chen: School of
		Mathematics and Statistics, Northeast Normal
		University, Changchun 130024, China}
	\email{chenly640@nenu.edu.cn}
	
	\begin{abstract}
		In this paper, we introduce the definition of multiplicative $\omega$-Lie bialgebra, which is equivalent to the Manin triples and matched pairs. We also study the $\omega$-Yang-Baxter equation and Yang-Baxter $\omega$-Lie bialgebra. The skew-symmetric solutions of the $\omega$-Yang-Baxter equation can be used to construct Yang-Baxter $\omega$-Lie bialgebra. We further introduce the concept of the $\omega$-$\mathcal{O}$-operator, which can be constructed from a left-symmetric algebras, and based on the $\omega$-$\mathcal{O}$-operator, we construct skew-symmetric solutions to the $\omega$-Yang--Baxter equation.

	   \end{abstract}

	\thanks{*Corresponding author.}
	\thanks{This work is supported by NNSF of
		China (Nos. 12271085, 12071405),  NSF of Jilin Province (No. YDZJ202201ZYTS589) and the Fundamental Research Funds for the Central Universities.}

	\maketitle

    \tableofcontents

    \allowdisplaybreaks
    \section{Introduction}
    In 2007, Nurowski \cite{cite23} introduced $\omega$-Lie algebras in the study of geometric problems. An $\omega$-Lie algebra satisfies a Jacobi identity twisted by a bilinear form $\omega$ , known as the $\omega$-Jacobi identity. Intuitively, $\omega$-Lie algebras can be viewed as a natural generalization of Lie algebras. In 2010, Zusmanovich \cite{cite34} provided significant work on the structure and representations of $\omega$-Lie algebras, playing a foundational role in the study of $\omega$-Lie algebras. In recent years, $\omega$-Lie (super) algebras have received much attention and have been  studied. This research includes classification theory \cite{cite13,cite15,cite10,cite31,cite32}, such as the classification of low-dimensional $\omega$-Lie (super) algebras and simple $\omega$-Lie algebras, as well as derivation theory \cite{cite14,cite16,cite24,cite25,cite32,cite33}, encompassing derivations, local derivations, 2-local derivations, biderivations, and generalized derivations. Moreover, the representation theory of $\omega$-Lie (super) algebras has also been investigated \cite{cite30,cite32}. Beyond $\omega$-Lie (super) algebras, other algebraic structures of this type have also been studied, such as $\omega$-left-symmetric algebras \cite{cite11,cite12}, which are $\omega$-generalizations of left-symmetric algebras. In recent years, the low-dimensional classification and related properties of $\omega$-left-symmetric algebras have been studied.
		
    A bialgebra is an algebraic structure equipped with both a multiplication and a comultiplication, satisfying certain compatibility conditions. Two famous  examples of bialgebras are associative bialgebras and Lie bialgebras. The significance of associative bialgebras is shows up in representation theory: the group algebras and the enveloping algebras of a Lie algebras both naturally carry the structure of an associative bialgebra, and this structural property serves as the algebraic foundation for the phenomenon that ``the tensor product of representations remains a representation" \cite{cite21,cite35}. This shows that associative bialgebras are not only natural algebraic structures but also a  bridge connecting algebraic structures with representation theory.  Lie bialgebras  can be viewed as a linearization of Poisson-Lie groups \cite{cite9,cite17}, having  deep backgrounds in  geometry and mathematical physics. Moreover, Lie
    bialgebras are closely related to the classical Yang-Baxter equation: a skew-symmetric solutions of the classical Yang-Baxter equation gives rise to a natural Lie algebra structure to the dual space of a Lie algebra, and the original Lie algebra becomes a Lie bialgebra. This connection not only reveals the construction of Lie bialgebras, but also shows why they are important in integrable systems and quantum group theory. In recent years, the bialgebra theories of various algebras have been extensive development. However, despite the profound geometric background and rich structural properties of $\omega$-Lie algebras, the bialgebra theory for $\omega$-Lie algebras has not yet been established. \textbf{The goal of this paper is to establish the bialgebra theory for $\omega$-Lie algebras.}
        
    A central problem in bialgebra theory is how to reasonably define a bialgebra structure for an algebra. Currently, most bialgebra theories for various types of algebras are based on analogues of Manin triples and matched pairs. The key to defining the compatibility conditions of a bialgebra structure via Manin triples and matched pairs lies in the adjoint and dual representations of the algebra itself. Not all algebras naturally admit adjoint or dual representations, especially the dual representation. In particular, the treatment of dual representations frequently leads to different bialgebra theories, such as those for Leibniz algebras \cite{cite1,cite2,cite28}, Hom-Lie algebras \cite{cite5,cite6,cite7,cite29}, and nearly associative algebras \cite{cite3,cite18}, each of these has more than one bialgebra theory. It seems that all algebras have adjoint representations, but $\omega$-Lie algebras are an exception. Their adjoint and dual representations have not been considered in the existing literature. For the adjoint representation, the standard construction does not give a representation. 
	
	In this paper, for an important class of $\omega$-Lie algebras, namely, the multiplicative $\omega$-Lie algebras, we provide their dual representation. To resolve the issue of the missing adjoint representation for $\omega$-Lie algebras, our strategy is to consider a more general algebraic structures, called generalized multiplicative $\omega$-Lie algebras, which extend multiplicative $\omega$-Lie algebras. We seek representations of generalized multiplicative $\omega$-Lie algebras and use them as generalized representations of multiplicative $\omega$-Lie algebras. This gives an analogue of the adjoint representation. Based on the above, we present a definition of multiplicative $\omega$-Lie bialgebra via Manin triples and matched pairs.
	
	The classical Yang-Baxter equation first appeared in the context of inverse scattering theory \cite{cite19}, which has played a prominent role in many fields, including integrable systems, quantum groups and quantum field theory, etc. \cite{cite8}. Semenov-Tian-Shansky \cite{cite27} discovered that, on quadratic Lie algebras, the classical Yang-Baxter equation in tensor form can be transformed into an operator form, namely $R: L \to L$, where $[R(x), R(y)] = R([R(x), y] + [x, R(y)])$, which is essentially an $\mathcal{O}$-operator corresponding to the adjoint representation. Subsequently, Kupershmidt \cite{cite22} found that the classical Yang-Baxter equation in tensor form on general Lie algebras can also be converted into an operator form, namely $R: L^* \to L$, where $[R(x), R(y)] = R(\mathrm{ad}^* R(x)(y) - \mathrm{ad}^* R(y)(x))$, which is essentially an $\mathcal{O}$-operator corresponding to the dual representation of the adjoint representation. Kupershmidt identified this phenomenon and proposed a general $\mathcal{O}$-operator corresponding to a representation $\rho$. A natural question arises: what is the relationship between the $\mathcal{O}$-operator corresponding to the representation $\rho$ and the Yang-Baxter equation? Bai \cite{cite4} addressed this question, discovering that the $\mathcal{O}$-operator corresponding to the representation $\rho$ is the operator form of the classical Yang-Baxter equation on a larger Lie algebra. Additionally, he investigated the relationship between left-symmetric algebras and $\mathcal{O}$-operators.
	
	For Lie algebras, it is standard that a Lie bialgebra, equivalent to the Manin triples and matched pairs of a Lie algebra, can be constructed starting from the skew-symmetric solutions of the classical Yang-Baxter equation. However, for $\omega$-Lie algebras, the skew-symmetric solutions of the $\omega$-Yang--Baxter equation do not construct a multiplicative $\omega$-Lie bialgebra. Furthermore, for Lie algebras, left-symmetric algebras can be used to construct the $\mathcal{O}$-operator corresponding to a representation of the Lie algebras. But in the case of $\omega$-Lie algebras, starting from $\omega$-left-symmetric algebras does not lead to the construction of the $\omega$-$\mathcal{O}$-operator.	
	
	In fact, we find that for $\omega$-Lie algebras, using skew-symmetric solutions of the $\omega$-Yang-Baxter equation gives a Yang-Baxter $\omega$-Lie bialgebra. The Yang-Baxter $\omega$-Lie bialgebra is not consistent with the multiplicative $\omega$-Lie bialgebra. However, when the multiplicative $\omega$-Lie algebra degenerates into a Lie algebra, both the Yang--Baxter $\omega$-Lie bialgebra and the multiplicative $\omega$-Lie bialgebra reduce to a Lie bialgebra. Furthermore, we discover that starting from $\omega$-left-symmetric algebras, we can construct an $\omega$-$\mathcal{O}$-operator corresponding to a Generalized Representation I. Additionally, we find that starting from left-symmetric algebras, we can construct an $\omega$-$\mathcal{O}$-operator.
	%When the multiplicative $\omega$-Lie algebra degenerates into a Lie algebra, the corresponding $\omega$-left-symmetric algebra also degenerates into a left-symmetric algebra, and both $\omega$-O-operators become O-operators on the Lie algebras.
	
	To make the structure of this work clear, we show the main parts of the paper and how they are connected in the figure below.
	
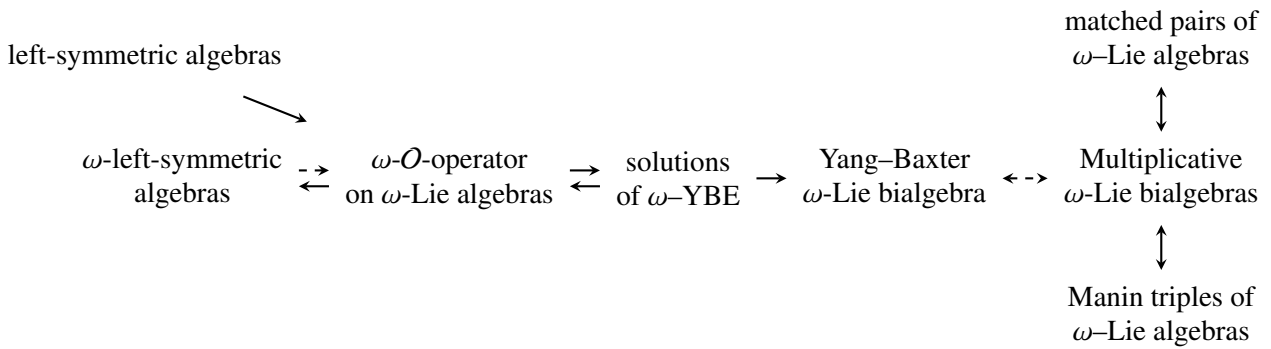
\begin{figure}[htbp]
		\centering
		\makebox[\textwidth]{% 强制居中，适应宽图
			\begin{tikzpicture}[
				node distance=0.6cm and 0.4cm, 
				align=center,
				every node/.style={font=\small, inner sep=6pt},
				>=stealth % 更现代的箭头
				]
				
				% 定义主链节点（水平排列）
				\node (lsa) {$\omega$-left-symmetric \\ algebras};
				\node[right=of lsa] (op) {$\omega$-$\mathcal{O}$-operator \\ on $\omega$-Lie algebras};
				\node[right=of op] (solutions) {solutions \\ of $\omega$--YBE};
				\node[right=of solutions] (ybe) {Yang--Baxter \\ $\omega$-Lie bialgebra};
				% 将 multiplicative 向右偏移 exactly 0.5cm from ybe
				\node[right=0.6cm of ybe] (multiplicative) {Multiplicative \\ $\omega$-Lie bialgebras};
				
				% 上方分支：经典 left-symmetric algebras
				\node[above=of lsa, xshift=-0.5cm] (cls) {left-symmetric algebras};
				
				% 下方分支：Manin triples
				\node[below=of multiplicative] (manin) {Manin triples of \\ $\omega$--Lie algebras};
				\node[above=of multiplicative] (matched) {matched pairs of \\ $\omega$--Lie algebras};
				
				% 绘制箭头
				
				% cls → op: 从经典到 ω-O-operator（斜线，表示推广）
				\draw[->, thick, shorten >=10pt, shorten <=10pt] (cls) -- (op);
				
				% lsa ↔ op: 双向构造（上下虚实线）
				\draw[->, thick, dashed] ([yshift=3pt]lsa.east) -- ([yshift=3pt]op.west);
				\draw[->, thick] ([yshift=-3pt]op.west) -- ([yshift=-3pt]lsa.east);
				
				% op ↔ solutions
				\draw[->, thick] ([yshift=3pt]op.east) -- ([yshift=3pt]solutions.west);
				\draw[->, thick] ([yshift=-3pt]solutions.west) -- ([yshift=-3pt]op.east);
				
				% solutions → ybe
				\draw[->, thick] (solutions) -- (ybe);
				
				% ybe ↔ multiplicative: 现在距离是 0.5cm
				\draw[<->, dashed, thick] (ybe) -- (multiplicative);
				
				% multiplicative ↔ matched & manin
				\draw[<->, thick] (multiplicative) -- (matched);
				\draw[<->, thick] (multiplicative) -- (manin);
				
			\end{tikzpicture}
		}
	\end{figure}
	
	The paper is organized as follows. In Section \ref{2}, we give some basic concepts and properties of $\omega$-Lie algebras, and particularly introduce the generalized multiplicative $\omega$-Lie algebras and their representations. In Section \ref{3}, we discuss the invariant bilinear form, Manin triples, and matched pairs for the multiplicative $\omega$-Lie algebra, and give the definition of a multiplicative $\omega$-Lie bialgebra. We prove the equivalence between the multiplicative $\omega$-Lie bialgebras, Manin triples, and matched pairs (Theorem \ref{38}). 
	In Section \ref{4}, we define the $\omega$-Yang-Baxter equation and the Yang-Baxter $\omega$-Lie bialgebra, constructing the Yang-Baxter $\omega$-Lie bialgebra using the skew-symmetric solutions of the $\omega$-Yang-Baxter equation (Theorem \ref{143}). Finally, in Section \ref{5}, we define the $\omega$-$\mathcal{O}$-operator and study its relationships with $\omega$-left-symmetric algebras (Theorem \ref{57}, Proposition \ref{510}) and left-symmetric algebras (Proposition \ref{515}). Furthermore, we use the $\omega$-$\mathcal{O}$-operator to construct a skew-symmetric solution to the $\omega$-Yang--Baxter equation in the semidirect product $\omega$-Lie algebra (Theorem \ref{518}).

	\section{$\omega$-Lie algebras and their representations}\label{2}
	In this section, we first review basic definitions of $\omega$-Lie algebras and their representations. Then, we introduce the dual representation and analogous adjoint representation of $\omega$-Lie algebras.
	\begin{definition}\cite{cite23}
        An anticommutative algebra \( L \) with multiplication \( [\cdot, \cdot] \) over a field \(\mathrm{K}\) is called an  $\omega$-Lie algebra if there exists a bilinear form \( \omega : L \times L \to\mathrm{K}\) such that
        \[
       [[x, y], z] + [[y, z], x] + [[z, x], y] = \omega(x, y)z + \omega(y, z)x + \omega(z, x)y
        \]
        for any \( x, y, z \in L \). We will refer to this identity as the \(\omega\)-Jacobi identity. 
   \end{definition}
   Clearly, an  $\omega$-Lie algebra is a natural generalization of a Lie algebra, reducing to a Lie algebra when \(\omega = 0\).
   \begin{definition}\cite{cite34}
		Let $(L, [\cdot, \cdot], \omega)$ be an $\omega$-Lie algebra over $\mathrm{K}$ and $V$ a vector space over $\mathrm{K}$. Then a representation of \( L \) on \( V \) is a linear map \( \rho : L \to \mathrm{gl}(V) \) such that
		\[
		\rho([x, y]) = \rho(x)\rho(y) - \rho(y)\rho(x) + \omega(x, y)\text{id}_V.
		\]
   \end{definition}
	By the definition of the representation of an $\omega$-Lie algebra, we can see that the usual adjoint map does not give a representation of an $\omega$-Lie algebra. In fact, the adjoint and dual representations of $\omega$-Lie algebras have not been considered or constructed in the existing literature. However, for a class of important $\omega$-Lie algebras, called multiplicative $\omega$-Lie algebras, we can give their dual representation.
	
	\begin{definition}\cite{cite34}
    	An $\omega$-Lie algebra is called multiplicative if there exists a linear form \( r : L \to \mathrm{K} \) such that
	    \[
    	[[x, y], z] + [[y, z], x] + [[z, x], y] = r([x, y])z + r([y, z])x + r([z, x])y
    	\]
	    for any \( x, y, z \in L \).
	\end{definition}
	It is straightforward to obtain the following conclusion.
    \begin{proposition}\label{21}
     	Let $(L, [\cdot, \cdot], r)$ be a multiplicative $\omega$-Lie algebra, and let \( (\rho, V) \) be a representation of \( L \). Define \( \rho^* : L \to \mathrm{gl}(V^*) \) by
    	$\rho^*(x)(\xi)(v) = -\xi(\rho(x)(v)) + 2r(x)\xi(v)$, for all \( x \in L \), \( \xi \in V^* \), and \( v \in V \). Then \( (\rho^*, V^*) \) is also a representation of \( L \).
   \end{proposition}
   
   We refer to \( (\rho^*,V^*) \) in Proposition \ref{21} as the dual representation of the representation \( (\rho,V) \).
   \begin{definition}
   	   Let $(L, [\cdot, \cdot], r)$ be a multiplicative $\omega$-Lie algebra, and let \( (\rho, V) \) be a representation of \( L \). Define \( \rho^* : L \to \mathrm{gl}(V^*) \) by
   	   $\rho^*(x)(\xi)(v) = -\xi(\rho(x)(v)) + 2r(x)\xi(v)$, for all \( x \in L \), \( \xi \in V^* \). The pair $(\rho^*, V^*)$ is called the dual representation of $(\rho, V)$.
   \end{definition}
   
   To resolve the issue of the missing adjoint representation for $\omega$-Lie algebras, we now introduce a more general algebraic structures: generalized multiplicative $\omega$-Lie algebras.
   \begin{definition}
       A vector space \( L \) with two multiplications \( [\cdot, \cdot]_1 \) and \( [\cdot, \cdot]_2 \) over a field \( \mathrm{K} \) is called a generalized multiplicative $\omega$-Lie algebra if there exists a linear form \( r : L \to \mathrm{K} \) such that
      \[
      [x, x]_1 = 0,
      \]
      \[
      [[x, y]_1, z]_2 + [[y, z]_1, x]_2 + [[z, x]_1, y]_2 = r([x, y]_1)z + r([y, z]_1)x + r([z, x]_1)y
      \]
      for any \( x, y, z \in L \).
   \end{definition}
   On a generalized multiplicative $\omega$-Lie algebra, we can define two types of representations, referred to as Representation I and Representation II.
   \begin{definition}
   	   Let $(L, [\cdot, \cdot]_1, [\cdot, \cdot]_2, r)$ be a generalized multiplicative $\omega$-Lie algebra over $\mathrm{K}$, and let $V$ be a vector space over $\mathrm{K}$. Then a Representation I of \( L \) on \( V \) consists of two linear maps \( \rho_1, \rho_2 : L \to \mathrm{gl}(V) \) such that
   	   \[
   	   \rho_1([x, y]_1) = \rho_2(x)\rho_1(y) - \rho_2(y)\rho_1(x) + r([x, y]_1)\text{id}_V.
   	   \]
   \end{definition}
   By a direct calculation, we have
   \begin{proposition}
   	   Let $(L, [\cdot, \cdot]_1, [\cdot, \cdot]_2, r)$ be a generalized multiplicative $\omega$-Lie algebra, and let $\rho_1, \rho_2 : L \to \mathrm{gl}(V)$ be two linear maps. Then $(\rho_1, \rho_2, V)$ is a Representation I of $L$ if and only if the direct sum $L \oplus V$ of vector spaces is a generalized multiplicative $\omega$-Lie algebra (the semi-direct product) by defining the multiplication on $L \oplus V$ as follows:
   	   \[
   	   [x + u, y + v]_{\overline{1}} = [x, y]_1 + \rho_1(x)(v) - \rho_1(y)(u),
   	   \]
   	   \[
   	   [x + u, y + v]_{\overline{2}} = [x, y]_2 + \rho_1(x)(v) - \rho_2(y)(u),
   	   \]
   	   \[
   	   \overline{r}(x + u) = r(x),
   	   \]
   	   for all $x, y \in L$ and $u, v \in V$.
   \end{proposition}
   \begin{definition}
       Let $(L, [\cdot, \cdot]_1, [\cdot, \cdot]_2, r)$ be a generalized multiplicative $\omega$-Lie algebra over $\mathrm{K}$, and let $V$ be a vector space over $\mathrm{K}$. Then a Representation II of \( L \) on \( V \) consists of two linear maps \( \rho_1, \rho_2 : L \to \mathrm{gl}(V) \) such that
       \[
       \rho_1([x, y]_1) = \rho_1(x)\rho_2(y) - \rho_1(y)\rho_2(x) + r([x, y]_1)\text{id}_V + 2r(x)\rho_1(y) - 2r(y)\rho_1(x) - 2r(x)\rho_2(y) + 2r(y)\rho_2(x).
       \]
       If there exists a linear map \( f : L \to \mathrm{gl}(V) \) such that
       \[
       f([x, y]_1) = 2r(x)\rho_1(y) - 2r(y)\rho_1(x) - 2r(x)\rho_2(y) + 2r(y)\rho_2(x),
       \]
       for all \( x, y \in L \), then the quadruple \( (\rho_1, \rho_2, f, V) \) is called a special Representation II.
   \end{definition}
   By a direct calculation, we get
   \begin{proposition}
   	    Let $(L, [\cdot, \cdot]_1, [\cdot, \cdot]_2, r)$ be a generalized multiplicative $\omega$-Lie algebra, and let $\rho_1, \rho_2, f : L \to \mathrm{gl}(V)$ be three linear maps. Then $(\rho_1, \rho_2, f, V)$ is a special Representation II of $L$ if and only if the direct sum $L \oplus V$ of vector spaces is a generalized multiplicative $\omega$-Lie algebra (the semi-direct product) by defining the multiplication on $L \oplus V$ as follows:
   	    \[
   	    [x + u, y + v]_{\overline{1}} = [x, y]_1 + \rho_2(x)(v) - \rho_2(y)(u),
   	    \]
   	    \[
   	    [x + u, y + v]_{\overline{2}} = [x, y]_2 + \rho_1(x)(v) - \rho_1(y)(u) - f(x)(v),
   	    \]
   	    \[
   	    \overline{r}(x + u) = r(x),
   	    \]
   	    for all $x, y \in L$ and $u, v \in V$.
   \end{proposition}
   
   Based on the above, we define two types of generalized representations for multiplicative $\omega$-Lie algebras through Representation I and Representation II of generalized multiplicative $\omega$-Lie algebras, referred to as Generalized Representation I and Generalized Representation II.
   \begin{definition}
   	   Let $(L, [\cdot, \cdot], r)$ be a multiplicative $\omega$-Lie algebra over \(\mathrm{K}\) and $V$ a vector space over \(\mathrm{K}\). Then a Generalized Representation I of $L$ on $V$ are two linear maps $\rho_1, \rho_2 : L \to \mathrm{gl}(V)$ such that
   	   \[
   	   \rho_1([x, y]) = \rho_2(x) \rho_1(y) - \rho_2(y) \rho_1(x) + r([x, y]) \mathrm{id}_V.
   	   \]
   \end{definition}
   \begin{definition}
       Let $(L, [\cdot, \cdot], r)$ be a multiplicative $\omega$-Lie algebra over \(\mathrm{K}\) and $V$ a vector space over \(\mathrm{K}\). Then a Generalized Representation II of $L$ on $V$ are two linear maps $\rho_1, \rho_2 : L \to \mathrm{gl}(V)$ such that
       \[
       \rho_1([x, y]) = \rho_1(x) \rho_2(y) - \rho_1(y) \rho_2(x) + r([x, y]) \mathrm{id}_V + 2r(x) \rho_1(y) - 2r(y) \rho_1(x) - 2r(x) \rho_2(y) + 2r(y) \rho_2(x).
       \]
   \end{definition}
   When $\rho_1 = \rho_2$, the Generalized Representations I and II reduce to representations of multiplicative $\omega$-Lie algebras. When the multiplicative $\omega$-Lie algebra degenerates to a Lie algebra, these representations degenerate to the representations of Lie algebras.
   
   From the generalized representations of multiplicative $\omega$-Lie algebras, we are now able to construct an adjoint-like representation for multiplicative $\omega$-Lie algebras, which we call the generalized adjoint representation.
   \begin{definition}
   	   Let $(L, [\cdot, \cdot], r)$ be a multiplicative $\omega$-Lie algebra. Define the maps $\mathrm{ad_1}, \mathrm{ad_2} : L \to \mathrm{gl}(L)$ by
       $
   	   \mathrm{ad_1}x(y) = [x, y], \mathrm{ad_2}x(y) = [x, y] + r(y)x,
   	   $
   	   for all $x, y \in L$. Then the quadruple $(\mathrm{ad_1}, \mathrm{ad_2}, L)$ is called the generalized adjoint representation of $L$.
   \end{definition}
   It is straightforward to obtain the following conclusion.
   \begin{proposition}
   	   Let $(L, [\cdot, \cdot], r)$ be a multiplicative $\omega$-Lie algebra. Then its generalized adjoint representation $(\mathrm{ad_1}, \mathrm{ad_2}, L)$ is a Generalized Representation I of $L$.
   \end{proposition}
   
   Although we have already constructed a dual representation for multiplicative $\omega$-Lie algebras, in order to place it within the broader framework of generalized representations, we slightly modify this dual representation and refer to it as the generalized dual representation.
   \begin{definition}
       Let $(L, [\cdot, \cdot], r)$ be a multiplicative $\omega$-Lie algebra, and let $(\rho_1, \rho_2, V)$ be a Generalized Representation I of $L$. Define the maps $\rho^*_1, \rho^*_2 : L \to \mathrm{gl}(V^*)$ by:
       \[
       \rho^*_1(x)(\xi)(v) = -\xi(\rho_1(x)(v)) + 2r(x)\xi(v),
       \]
       \[
       \rho^*_2(x)(\xi)(v) = -\xi(\rho_2(x)(v)) + 2r(x)\xi(v),
       \]
       for all $x \in L$, $\xi \in V^*$, and $v \in V$. Then the triple $(\rho^*_1, \rho^*_2, V^*)$ is called the generalized dual representation of $(\rho_1, \rho_2, V)$.
   \end{definition}
   By a direct calculation, we have
   \begin{proposition}
   	   Let $(L, [\cdot, \cdot], r)$ be a multiplicative $\omega$-Lie algebra, and let $(\rho_1, \rho_2, V)$ be a Generalized Representation I of $L$. Then its generalized dual representation $(\rho^*_1, \rho^*_2, V^*)$ is a Generalized Representation II of $L$.
   \end{proposition}
   
   Furthermore, we consider a special type of Generalized Representation II, referred to as the Associated Generalized Representation II.
   \begin{definition}
   	   Let $(L, [\cdot, \cdot], r)$ be a multiplicative $\omega$-Lie algebra over \( \mathrm{K} \) and $L^*$ is the dual space of $L$. Then an Associated Generalized Representation II of \( L \) on \( L^* \) is a Generalized Representation II $(\rho_1, \rho_2, L^*)$  satisfying the following condition:
   	   \[
   	   \rho_2(x)(\xi) = \rho_1(x)(\xi) - \xi(x)r.
   	   \]
   \end{definition}
   It is straightforward to get the following conclusion.
   \begin{proposition}
       Let $(L, [\cdot, \cdot], r)$ be a multiplicative $\omega$-Lie algebra. Then the generalized dual representation of its generalized adjoint representation is an Associated Generalized Representation II.
   \end{proposition}
   
   \section{Matched pairs, Manin triples and Multiplicative $\omega$-Lie bialgebras}\label{3}
   In this section, we will introduce the Manin triples and matched pairs of multiplicative $\omega$-Lie algebras, and give the definition of a multiplicative $\omega$-Lie bialgebra, proving the equivalence among the three.
   
  % We first introduce the concepts related to the matched pairs of multiplicative $\omega$-Lie algebras.
   
   Let $(L, [\cdot, \cdot]_L, r)$ and $(L^*, [\cdot, \cdot]_{L^*}, r^*)$ be two multiplicative $\omega$-Lie algebras, where $L^*$ is the dual space of $L$. Let $(\rho_1, \rho_2, L^*)$ and $(\rho^*_1, \rho^*_2, L)$ be Associated Generalized Representation II of $L$ and $L^*$ respectively. On the direct sum $L \oplus L^*$ of the underlying vector spaces, define
   \[
   r_{L \oplus L^*} : L \oplus L^* \to F, \quad r_{L \oplus L^*}(x + \xi) = r(x) + r^*(\xi),
   \]
   and define a bilinear map \([\cdot, \cdot]_{L \oplus L^*} : \otimes^2 (L \oplus L^*) \to L \oplus L^*\) by
   \[
   [x + \xi, y + \eta]_{L \oplus L^*} = ([x, y]_L - \rho^*_2(\eta)(x) + \rho^*_1(\xi)(y) - \xi(y) u_r) + ([\xi, \eta]_{L^*} + \rho_1(x)(\eta) - \rho_2(y)(\xi) - \eta(x) r),
   \]
   for all $x, y \in L$ and $\xi, \eta \in L^*$, where $u_r \in L$ satisfies
   \[
   \langle r^*, a \rangle = \langle a, u_r \rangle, \quad \forall\, a \in L^*.
   \]
   \begin{theorem}\label{31}
   	   Under the above notation, $(L \oplus L^*, [\cdot, \cdot]_{L \oplus L^*}, r_{L \oplus L^*})$ is a multiplicative $\omega$-Lie algebra if and only if the following conditions hold:
   	   \begin{enumerate}
   	   	\item 
   	   	$\rho^*_2(\omega)([x, y]_L) = [\rho^*_2(\omega)(x), y]_L + [x, \rho^*_2(\omega)(y)]_L + \rho^*_1(\rho_2(y)(\omega))(x) - \rho^*_1(\rho_2(x)(\omega))(y)  \\
   	   	\quad + (\rho_2(x)(\omega))(y)u_r - (\rho_2(y)(\omega))(x)u_r + r(\rho^*_2(\omega)(y))x - r(\rho^*_2(\omega)(x))y + r^*(\rho_2(x)(\omega))y \\
   	   	\quad - r^*(\rho_2(y)(\omega))x$,
   	   	
   	   	\item 
   	   	$[\eta, \xi]_{L^*}(z)u_r = 2r^*(\xi)\rho^*_2(\eta)(z) + 2r^*(\eta)\rho^*_1(\xi)(z) - 2r^*(\xi)\rho^*_1(\eta)(z) - 2r^*(\eta)\rho^*_2(\xi)(z) \\
   	   	\quad - \xi(\rho^*_2(\eta)(z))u_r + \eta(\rho^*_2(\xi)(z))u_r$,
   	   	
   	   	\item 
   	   	$\rho_2(z)([\xi, \eta]_{L^*}) = [\rho_2(z)(\xi), \eta]_{L^*} + [\xi, \rho_2(z)(\eta)]_{L^*} + \rho_1(\rho^*_2(\eta)(z))(\xi) - \rho_1(\rho^*_2(\xi)(z))(\eta) \\
   	   	\quad + \eta(\rho^*_2(\xi)(z))r - \xi(\rho^*_2(\eta)(z))r + r^*(\rho_2(z)(\eta))\xi - r^*(\rho_2(z)(\xi))\eta + r(\rho^*_2(\xi)(z))\eta - r(\rho^*_2(\eta)(z))\xi$,
   	   	
   	   	\item 
   	   	$\omega([y, x]_L)r = 2r(x)\rho_2(y)(\omega) + 2r(y)\rho_1(x)(\omega) - 2r(x)\rho_1(y)(\omega) - 2r(y)\rho_2(x)(\omega) \\
   	   	\quad - \rho_2(y)(\omega)(x)r + \rho_2(x)(\omega)(y)r$.
   	   \end{enumerate}
   	   for all $x, y, z \in L$ and $\xi, \eta, \omega \in L^*$.
   \end{theorem}
   \begin{proof}
       The verification is a routine calculation and is therefore omitted.
   \end{proof}
   \begin{definition}
        Let $(L, [\cdot, \cdot], r)$ be a multiplicative $\omega$-Lie algebra. A matched pair of multiplicative $\omega$-Lie algebras is a quadruple \((L, L^*; (\rho_1, \rho_2), (\rho^*_1, \rho^*_2))\) consisting of two multiplicative $\omega$-Lie algebras $L$ and $L^*$, together with Associated Generalized Representation II $(\rho_1, \rho_2, L^*)$ and $(\rho^*_1, \rho^*_2, L)$ such that the four conditions in Theorem \ref{31} are satisfied.
   \end{definition}
   When the multiplicative $\omega$-Lie algebra degenerates to a Lie algebra, the above matched pair reduces to a matched pair of Lie algebras.
   
   We now introduce concepts related to Manin triples for multiplicative $\omega$-Lie algebras.
   \begin{definition}
       A bilinear form \( B : L \times L \to F \) on a multiplicative $\omega$-Lie algebra $(L, [\cdot, \cdot], r)$ is called invariant if it satisfies the following condition:
       \[
       B([x, y], z) = B(x, [y, z]) - 2r(y)B(x, z) + r(x)B(y, z) + r(z)B(x, y),
       \]
       for all \( x, y, z \in L \).
   \end{definition}
   When the $\omega$-Lie algebra degenerates to a Lie algebra, the above invariance condition reduces precisely to the standard invariance condition for a bilinear form on a Lie algebra.
   \begin{definition}
       A Manin triple of multiplicative $\omega$-Lie algebras is a triple of multiplicative $\omega$-Lie algebras $(H; L, L')$ together with a nondegenerate symmetric invariant bilinear form $B(\cdot, \cdot)$ on $H$ such that $L$ and $L'$ are isotropic multiplicative $\omega$-Lie subalgebras of $H$, i.e.,$\forall\, x, y \in L, B(x, y) = 0,\forall\, x', y' \in L', B(x', y') = 0$, and $H = L \oplus L'$ as a vector space.
   \end{definition}
   \begin{theorem}\label{35}
   	   Let $(L, [\cdot, \cdot]_L, r)$ and $(L^*, [\cdot, \cdot]_{L^*}, r^*)$ be two multiplicative $\omega$-Lie algebras. Then the triple $(L \oplus L^*; L, L^*)$ is a Manin triple of multiplicative $\omega$-Lie algebras associated to the nondegenerate symmetric invariant bilinear form on $L \oplus L^*$ defined by 
   	   \begin{equation}
   	   	 \mathrm{B}\left( x + a, y + b \right) \coloneqq \langle b, x \rangle + \langle a, y \rangle \tag{3.1} \label{123}
   	   \end{equation}
   	   if and only if $(L, L^*; (\mathrm{ad}^*_1, \mathrm{ad}^*_2), (\mathrm{AD}^*_1, \mathrm{AD}^*_2))$ is a matched pair of multiplicative $\omega$-Lie algebras. Here, $(\mathrm{ad}^*_1, \mathrm{ad}^*_2)$ denotes the generalized dual representation of the generalized adjoint representation of $L$, and $(\mathrm{AD}^*_1, \mathrm{AD}^*_2)$ denotes the corresponding one for $L^*$.
   \end{theorem}
   \begin{proof}
   	   	Suppose \((L, L^*, (\mathrm{ad}_1^*, \mathrm{ad}_2^*), (\mathrm{AD}_1^*, \mathrm{AD}_2^*))\) is a matched pair of multiplicative \(\omega\)-Lie algebras. Then \((L \oplus L^*, [\cdot, \cdot]_d, r_d)\) is a multiplicative \(\omega\)-Lie algebra, where \([\cdot, \cdot]_d\) and \(r_d\) are defined as:
   	   	\begin{align*}
   	   		& [x + a, y + b]_d = [x, y]_L - \mathrm{AD}_2^*(b)(x) + \mathrm{AD}_1^*(a)(y) - \alpha(y)u_r  + [a, b]_{L^*} + \mathrm{ad}_1^*(x)(b) - \mathrm{ad}_2^*(y)(a) - b(x)r,
   	   	\end{align*}
   	   	\[
   	   	r_d(x + a) = r(x) + r^*(a).
   	   	\]
   	   	Let \(x, y, z \in L\) and \(a, b, c \in L^*\). We only need to prove that the bilinear form defined by equation (\ref{123}) is invariant. We have
   	   	\begin{align*}
   	   		& B([x, y]_d, z) = 0 = B(x, [y, z]_d) - 2r_d(y)B(x, z) + r_d(y)B(x, z) + r_d(z)B(x, y). 
   	   	\end{align*}
   	   	Note that
   	   	\begin{align*}
   	   		& B([x, y]_d, c) = B([x, y]_L, c) = \langle c, [x, y]_L \rangle \\
   	   		& = -\langle c, [y, x]_L \rangle - r(x)\langle c, y \rangle + 2r(y)\langle c, x \rangle + r(x)\langle c, y \rangle - 2r(y)\langle c, x \rangle \\
   	   		& = \langle \mathrm{ad}_2^*(y)(c), x \rangle + r(x)\langle c, y \rangle - 2r(y)\langle c, x \rangle,
   	   	\end{align*}
   	   	we have
   	   	\begin{align*}
   	   		& B(x, [y, c]_d) - 2r_d(y)B(x, c) + r_d(c)B(x, y) + r_d(x)B(y, c) \\
   	   		& = B(x, -\mathrm{AD}_2^*(c)(y) + \mathrm{ad}_1^*(y)(c) - c(y)r) - 2r(y)B(x, c) + r^*(c)B(x, y) + r(x)B(y, c) \\
   	   		& = \langle \mathrm{ad}_1^*(y)(c) - c(y)r, x \rangle - 2r(y)\langle c, x \rangle + r(x)\langle c, y \rangle \\
   	   		& = \langle \mathrm{ad}_1^*(y)(c), x \rangle - 2r(y)\langle c, x \rangle + r(x)\langle c, y \rangle = B([x, y]_d, c).
   	   	\end{align*}
   	   	Similarly, a direct computation shows that
   	   	\begin{align*}
   	   		& B([x, b]_d, z) = B(-\mathrm{AD}_2^*(b)(x) + \mathrm{ad}_1^*(x)(b) - b(x)r, z) \\
   	   		& = \langle \mathrm{ad}_1^*(x)(b) - b(x)r, z \rangle  = \langle \mathrm{ad}_2^*(x)(b), z \rangle,
   	   	\end{align*}
   	   	therefore, we obtain	
   	   	\begin{align*}
   	   		& B(x, [b, z]_d) - 2r_d(b)B(x, z) + r_d(z)B(x, b) + r_d(x)B(b, z) \\
   	   		& = B(x, \mathrm{AD}_1^*(b)(z) - b(z)u_r - \mathrm{ad}_2^*(z)(b) + r(z)B(x, b)) + r(x)B(b, z) \\
   	   		& = -\langle \mathrm{ad}_2^*(z)(b), x \rangle + r(z)\langle b, x \rangle + r(x)\langle b, z \rangle \\
   	   		& = -(-b, [z, x]_L + r(z)x + 2r(z)\langle b, x \rangle) + r(z)\langle b, x \rangle + r(x)\langle b, z \rangle \\
   	   		& = \langle b, [z, x]_L \rangle - \langle b, r(z)x \rangle + 2r(x)\langle b, z \rangle  = -\langle b, [x, z]_L \rangle - \langle b, r(z)x \rangle + 2r(x)\langle b, z \rangle \\
   	   		& = \langle \mathrm{ad}_2^*(x)(b), z \rangle = B([x, b]_d, z),
   	   	\end{align*}
   	   	and
   	   	\begin{align*}
   	   		& B([x, b]_d, c) = B(-\mathrm{AD}_2^*(b)(x) + \mathrm{ad}_1^*(x)(b) - b(x)r, c)  = -\langle c, \mathrm{AD}_2^*(b)(x) \rangle. 
   	   	\end{align*}
   	   	It follows that
   	   	\begin{align*}
   	   		& B(x, [b, c]_d) - 2r_d(b)B(x, c) + r_d(c)B(x, b) + r_d(x)B(b, c) \\
   	   		& = B(x, [b, c]_*) - 2r^*(b)B(x, c) + r^*(c)B(x, b) + r(x)B(b, c) \\
   	   		& = \langle [b, c]_*, x^* \rangle - 2r^*(b)\langle c, x \rangle + r^*(c)\langle b, x \rangle + r(x)\langle b, c \rangle \\
   	   		& = -\langle c, \mathrm{AD}_2^*(b)(x) \rangle = B([x, b]_d, c).
   	   	\end{align*}
   	   	In the same way, we have
   	   	\begin{align*}
   	   		& B([a, b]_d, z) = B([a, b]_*, z) = \langle [a, b]_*, z^* \rangle = -\langle [b, a]_*, z^* \rangle \\
   	   		& = -\langle [b, a]_*, z^* \rangle - r^*(a)\langle b, z \rangle + 2r^*(b)\langle a, z \rangle + r^*(a)\langle b, z \rangle - 2r^*(b)\langle a, z \rangle \\
   	   		& = \langle a, \mathrm{AD}_2^*(b)(z) \rangle - 2r^*(b)\langle a, z \rangle + r^*(a)\langle b, z \rangle, 
   	   	\end{align*}
   	   	hence, we conclude that
   	   	\begin{align*}
   	   		& B(a, [b, z]_d) - 2r_d(b)B(a, z) + r_d(z)B(a, b) + r_d(a)B(b, z) \\
   	   		& = B(a, A\mathrm{d}_1^*(b)(z) - b(z)u_r - \mathrm{ad}_2^*(z)(b)) - 2r^*(b)B(a, z) + r^*(a)B(b, z) \\
   	   		& = -\langle a, \mathrm{AD}_1^*(b)(z) - b(z)u_r \rangle - 2r^*(b)\langle a, z \rangle + r^*(a)\langle b, z \rangle \\
   	   		& = \langle a, \mathrm{AD}_2^*(b)(z) \rangle - 2r^*(b)\langle a, z \rangle + r^*(a)\langle b, z \rangle = B([a, b]_d, z).
   	   	\end{align*}
   	   	Therefore, the bilinear form defined by equation (\ref{123}) is invariant. 
   	   	
   	   	Conversely, let \((L \oplus L^*, L, L^*)\) be the Manin triple of the \(\omega\)-Lie algebra associated with the invariant bilinear form \(B\) given by equation (\ref{123}). Then for any \(x, y, z \in L\) and \(a, b, c \in L^*\), by the invariance of \(B\), we have
   	   	\[
   	   	B([x + a, y + b]_d, z + c)
   	   	= B(x + a, [y + b, z + c]_d) -2r_d(y + b)B(x + a, z + c)
   	   	\]
   	   	\[
   	   	 + r_d(x + a)B(y + b, z + c) + r_d(z + c)B(x + a, y + b).
   	   	\]
   	   	Note that
   	   	\begin{align*}
   	   		& B([x, b]_d, z) = -B([b, x]_d, z) \\
   	   		& = -B(b, [x, z]_d) + 2r_d(x)B(b, z) - r_d(z)B(b, x) + r_d(c)B(x, z) \\
   	   		& = -\bigl( B(b, [x, z]_L) + 2r(x)B(b, z) + r(z)B(b, x) \bigr) \\
   	   		& = -\langle b, [x, z]_L \rangle + 2r(x)\langle b, z \rangle + r(z)\langle b, x \rangle \\
   	   		& = -\bigl( \langle b, [x, z]_L + r(z)x \rangle - 2r(x)\langle b, z \rangle \bigr)  = \langle \mathrm{ad}_2^*x(b), z \rangle,
   	   	\end{align*} 
   	   	we get
   	   	\begin{align*}
   	   		& B(-\mathrm{AD}_2^*b(x) + \mathrm{ad}_1^*x(b) - b(x)r, z) = \langle \mathrm{ad}_1^*x(b) - b(x)r, z \rangle\\
   	   		&  = \langle \mathrm{ad}_2^*x(b), z \rangle=B([x, b]_d, z).
   	   	\end{align*}
   	   	Similarly, a direct computation shows that
   	   	\begin{align*}
   	   		& B([x, b]_d, c) = B(x, [b, c]_d) - 2r_d(b)B(x, c) + r_d(c)B(x, b) + r_d(x)B(b, c) \\
   	   		& = B(x, [b, c]_{L^*}) - 2r^*(b)B(x, c) + r^*(c)B(x, b) \\
   	   		& = \langle [b, c]_{L^*}, x \rangle + \langle r^*(c)b, x \rangle - 2r^*(b)\langle c, x \rangle = -\langle c, \mathrm{AD}_2^*b(x) \rangle,
   	   	\end{align*}
   	   	therefore, we obtain
   	   	\begin{align*}
   	   		& B(-\mathrm{AD}_2^*b(x) + \mathrm{ad}_1^*x(b) - b(x)r, c)  = \langle c, -\mathrm{AD}_2^*b(x) \rangle\\
   	   		&  = -\langle c, \mathrm{AD}_2^*b(x) \rangle=B([x, b]_d, c).
   	   	\end{align*}   	   	
   	    We conclude that $[x, b]_d = -\mathrm{AD}_2^*b(x) + \mathrm{ad}_1^*x(b) - b(x)r$. Therefore, the multiplicative \(\omega\)-Lie bracket on \(L \oplus L^*\) is given by $[x + a, y + b]_d = [x, y]_L - \mathrm{AD}_2^*(b)(x) + \mathrm{AD}_1^*(a)(y) - \alpha(y)u_r  + [a, b]_{L^*} + \mathrm{ad}_1^*(x)(b) - \mathrm{ad}_2^*(y)(a) - b(x)r$. By Theorem \ref{31}, $\bigl(L, L^*, (\mathrm{ad}_1^*, \mathrm{ad}_2^*), (\mathrm{AD}_1^*, \mathrm{AD}_2^*)\bigr)$ constitutes a matched pair of multiplicative $\omega$-Lie algebras.
   \end{proof}
   For a multiplicative $\omega$-Lie algebra $(L^*, [\cdot, \cdot]_{L^*}, r^*$), 
   let $\Delta : L \to L \otimes L$ be the dual map of 
  $[\cdot, \cdot]_{L^*} : L^* \otimes L^* \to L^*$, i.e.,
   \begin{align*}
   	\langle \xi \otimes \eta, \Delta(x) \rangle &= \langle [\xi, \eta]_{L^*} + r^*(\eta)\xi - 2r^*(\xi)\eta, x \rangle.
   \end{align*}
   \begin{theorem}\label{36}
   	  Let \((L, [\cdot, \cdot]_L, r)\) and \((L^*, [\cdot, \cdot]_{L^*}, r^*)\) be two multiplicative \(\omega\)-Lie algebras.Then the quadruple \((L, L^*; (ad^*_1, ad^*_2), (AD^*_1, AD^*_2))\) is a matched pair if and only if the following two conditions hold:
   	  \begin{enumerate}
   	  	\item $\Delta([x, y]_L) = (\mathrm{id} \otimes \mathrm{ad}_2 x + \mathrm{ad}_2 x \otimes \mathrm{id})\Delta(y) - (\mathrm{id} \otimes \mathrm{ad}_2 y + \mathrm{ad}_2 y \otimes \mathrm{id})\Delta(x) + 2r(y)\Delta(x) - 2r(x)\Delta(y) \\
   	  	\quad + \bigl(\mathrm{ad}_2(x)(u_r) - 2r(x)u_r\bigr) \otimes y - \bigl(\mathrm{ad}_2(y)(u_r) - 2r(y)u_r\bigr) \otimes x,$
   	  	
   	  	\item $(\xi \otimes \eta) \Delta(z) u_r  = - 2r^*(\xi)\mathrm{AD}_2^*(\eta)(z) - 2r^*(\eta)\mathrm{AD}_1^*(\xi)(z) + 2r^*(\xi)\mathrm{AD}_1^*(\eta)(z) + 2r^*(\eta)\mathrm{AD}_2^*(\xi)(z) \\
   	  	\quad + \xi(\mathrm{AD}_2^*(\eta)(z))u_r - \eta(\mathrm{AD}_2^*(\xi)(z))u_r + r^*(\eta)(\xi)(z) u_r - 2r^*(\xi)(\eta)(z) u_r.$
   	  \end{enumerate}
   	  for all $x, y, z \in L$ and $\xi, \eta \in L^*$, where $u_r \in L$ satisfies
   	  $\langle r^*, a \rangle = \langle a, u_r \rangle, \forall\, a \in L^*$.
   \end{theorem}
   \begin{proof}
   	By Theorem \ref{31}, \((L, L^*; (\mathrm{ad}_1^*, \mathrm{ad}_2^*), (\mathrm{AD}_1^*, \mathrm{AD}_2^*))\) is a matched pair of multiplicative \(\omega\)-Lie algebras if and only if
   	\begin{align*}
   		&\mathrm{AD}_2^*(a)([x, y]_L) = [\mathrm{AD}_2^*(a)(x), y]_L + [x, \mathrm{AD}_2^*(a)(y)]_L + \mathrm{AD}_1^*\bigl(\mathrm{ad}_2^*(y)(a)\bigr)(x) - \mathrm{AD}_1^*\bigl(\mathrm{ad}_2^*(x)(a)\bigr)(y) \\
   		&\quad + \bigl(\mathrm{ad}_2^*(x)(a)\bigr)(y)u_r - \bigl(\mathrm{ad}_2^*(y)(a)\bigr)(x)u_r + r\bigl(\mathrm{AD}_2^*(a)(y)\bigr)(x) - r\bigl(\mathrm{AD}_2^*(a)(x)\bigr)(y) \\
   		&\quad + r^*\bigl(\mathrm{ad}_2^*(x)(a)\bigr)(y) - r^*\bigl(\mathrm{ad}_2^*(y)(a)\bigr)(x), \\
   		&([\xi, \eta]_{L^*})(z)u_r = -2r^*(\xi)\mathrm{AD}_2^*(\eta)(z) - 2r^*(\eta)\mathrm{AD}_1^*(\xi)(z) + 2r^*(\xi)\mathrm{AD}_1^*(\eta)(z) + 2r^*(\eta)\mathrm{AD}_2^*(\xi)(z) \\
   		&\quad + \xi\bigl(\mathrm{AD}_2^*(\eta)(z)\bigr)u_r - \eta\bigl(\mathrm{AD}_2^*(\xi)(z)\bigr)u_r.
   	\end{align*}
   	Then we get
   	\begin{align*}
   		0 &= \bigl\langle -\mathrm{AD}_2^*(a)([x, y]_L) + [\mathrm{AD}_2^*(a)(x), y]_L - r(\mathrm{AD}_2^*(a)(x))(y) + [x, \mathrm{AD}_2^*(a)(y)]_L + r(\mathrm{AD}_2^*(a)(y))(x)  \\
   		& \quad + \mathrm{AD}_1^*(\mathrm{ad}_2^*(y)(a))(x) - (\mathrm{ad}_2^*(y)(a))(x)u_r - \mathrm{AD}_1^*(\mathrm{ad}_2^*(x)(a))(y) + (\mathrm{ad}_2^*(x)(a))(y) u_r + r^*\bigl(\mathrm{ad}_2^*(x)(a)\bigr)y \\
   		& \quad - r^*\bigl(\mathrm{ad}_2^*(y)(a)\bigr)x,\ b \bigr\rangle \\
   		&= \langle [x, y]_L,\ [a, b]_{L^*} + r^*(b)a - 2r^*(a)b \rangle \\
   		&\quad - \bigl\langle [y, \mathrm{AD}_2^*(a)(x)]_L + r(\mathrm{AD}_2^*(a)(x))y - 2r(y)\mathrm{AD}_2^*(a)(x),\ b \bigr\rangle - 2r(y)\langle \mathrm{AD}_2^*(a)(x),\ b \rangle \\
   		&\quad + \bigl\langle [x, \mathrm{AD}_2^*(a)(y)]_L + r(\mathrm{AD}_2^*(a)(y))x - 2r(x)\mathrm{AD}_2^*(a)(y),\ b \bigr\rangle + 2r(x)\langle \mathrm{AD}_2^*(a)(y),\ b \rangle \\
   		&\quad + \bigl\langle \mathrm{AD}_2^*\bigl(\mathrm{ad}_2^*(y)(a)\bigr)(x),\ b \bigr\rangle - \bigl\langle \mathrm{AD}_2^*\bigl(\mathrm{ad}_2^*(x)(a)\bigr)(y),\ b \bigr\rangle + \langle \mathrm{ad}_2^*(x)(a)(u_r)(y), b \rangle - \langle \mathrm{ad}_2^*(y)(a)(u_r)(x), b \rangle \\
   		&= \langle [x, y]_L,\ [a, b]_{L^*} + r^*(b)a - 2r^*(a)b \rangle - \bigl\langle x,\ [a, \mathrm{ad}_2^*(y)(b)]_{L^*} + r^*(\mathrm{ad}_2^*(y)(b))a - 2r^*(a)\mathrm{ad}_2^*(y)(b) \bigr\rangle \\
   		&\quad - 2r(y)\langle x,\ [a, b]_{L^*} + r^*(b)a - 2r^*(a)b \rangle + \bigl\langle y,\ [a, \mathrm{ad}_2^*(x)(b)]_{L^*} + r^*(\mathrm{ad}_2^*(x)(b))a - 2r^*(a)\mathrm{ad}_2^*(x)(b) \bigr\rangle \\
   		&\quad + 2r(x)\langle y,\ [a, b]_{L^*} + r^*(b)a - 2r^*(a)b \rangle - \bigl\langle x,\ [\mathrm{ad}_2^*(y)(a), b]_{L^*} + r^*(b)\mathrm{ad}_2^*(y)(a) - 2r^*(\mathrm{ad}_2^*(y)(a))b \bigr\rangle \\
   		&\quad + \bigl\langle y,\ [\mathrm{ad}_2^*(x)(a), b]_{L^*} + r^*(b)\mathrm{ad}_2^*(x)(a) - 2r^*(\mathrm{ad}_2^*(x)(a))b \bigr\rangle +\left( -\langle \mathrm{ad}_2(x)(u_r), a \rangle + 2r(x)\langle u_r, a \rangle \right) \langle y, b \rangle \\
   		& \quad -\left( -\langle \mathrm{ad}_2(y)(u_r), a \rangle + 2r(y)\langle u_r, a \rangle \right) \langle x, b \rangle \\
   		&= \langle \Delta([x, y]_L),\ a \otimes b \rangle - \langle \Delta(x),\ a \otimes \mathrm{ad}_2^*(y)(b) \rangle - 2r(y)\langle \Delta(x),\ a \otimes b \rangle + \langle \Delta(y),\ a \otimes \mathrm{ad}_2^*(x)(b) \rangle \\
   		&\quad + 2r(x)\langle \Delta(y),\ a \otimes b \rangle - \langle \Delta(x),\ \mathrm{ad}_2^*(y)(a) \otimes b \rangle + \langle \Delta(y),\ \mathrm{ad}_2^*(x)(a) \otimes b \rangle -\langle \mathrm{ad}_2(x)(u_r), a \rangle \langle y, b \rangle \\
   		& \quad + 2r(x)\langle u_r, a \rangle \langle y, b \rangle
   		+\langle \mathrm{ad}_2(y)(u_r), a \rangle \langle x, b \rangle - 2r(y)\langle u_r, a \rangle \langle x, b \rangle \\
   		&= \langle \Delta([x, y]_L),\ a \otimes b \rangle - \langle \Delta(x),\ (\mathrm{id} \otimes \mathrm{ad}_2^*(y))(a \otimes b) \rangle - 2r(y)\langle \Delta(x),\ a \otimes b \rangle + \langle \Delta(y),\ (\mathrm{id} \otimes \mathrm{ad}_2^*(x))(a \otimes b) \rangle \\
   		&\quad  + 2r(x)\langle \Delta(y),\ a \otimes b \rangle - \langle \Delta(x),\ (\mathrm{ad}_2^*(y) \otimes \mathrm{id})(a \otimes b) \rangle + \langle \Delta(y),\ (\mathrm{ad}_2^*(x) \otimes \mathrm{id})(a \otimes b) \rangle \\
   		& \quad + \langle \bigl(-\mathrm{ad}_2(x)(u_r) + 2r(x)u_r\bigr) \otimes y,\ a \otimes b \rangle
   		+ \langle \bigl(\mathrm{ad}_2(y)(u_r) - 2r(y)u_r\bigr) \otimes x,\ a \otimes b \rangle \\
   		&= \langle \Delta([x, y]_L),\ a \otimes b \rangle + \langle (\mathrm{id} \otimes \mathrm{ad}_2(y))\Delta(x),\ a \otimes b \rangle - \langle (\mathrm{id} \otimes \mathrm{ad}_2(x))\Delta(y),\ a \otimes b \rangle \\
   		&\quad + \langle (\mathrm{ad}_2(y) \otimes \mathrm{id})\Delta(x),\ a \otimes b \rangle - 2r(y)\langle \Delta(x),\ a \otimes b \rangle - \langle (\mathrm{ad}_2(x) \otimes \mathrm{id})\Delta(y),\ a \otimes b \rangle + 2r(x)\langle \Delta(y),\ a \otimes b \rangle \\
   		& \quad + \langle \bigl(-\mathrm{ad}_2(x)(u_r) + 2r(x)u_r\bigr) \otimes y,\ a \otimes b \rangle
   		+ \langle \bigl(\mathrm{ad}_2(y)(u_r) - 2r(y)u_r\bigr) \otimes x,\ a \otimes b \rangle.
   	\end{align*}
   	which is exactly 
   	\[
   	\Delta([x, y]_L) = (\mathrm{id} \otimes \mathrm{ad}_2 x + \mathrm{ad}_2 x \otimes \mathrm{id})\Delta(y) - (\mathrm{id} \otimes \mathrm{ad}_2 y + \mathrm{ad}_2 y \otimes \mathrm{id})\Delta(x) + 2r(y)\Delta(x) - 2r(x)\Delta(y)
   	\]
   	\[
   	+ \bigl(\mathrm{ad}_2(x)(u_r) - 2r(x)u_r\bigr) \otimes y - \bigl(\mathrm{ad}_2(y)(u_r) - 2r(y)u_r\bigr) \otimes x.
   	\]
   	We have
   	$([\xi, \eta]_{L^*})(z) u_r  = \left(\left([\xi, \eta]_{L^*}(z) + r^*(\eta)(\xi)(z) - 2r^*(\xi)(\eta)\right)(z) - r^*(\eta)(\xi)(z) +2r^*(\xi)(\eta)(z)\right) u_r$ 
   	\[= (\xi \otimes \eta) \Delta(z) u_r - r^*(\eta)(\xi)(z) u_r + 2r^*(\xi)(\eta)(z) u_r.\]
   	Hence we get
   	\[
   	(\xi \otimes \eta) \Delta(z) u_r - r^*(\eta)(\xi)(z) u_r + 2r^*(\xi)(\eta)(z) u_r  = - 2r^*(\xi)\mathrm{AD}_2^*(\eta)(z) - 2r^*(\eta)\mathrm{AD}_1^*(\xi)(z)
   	\]
   	\[+ 2r^*(\xi)\mathrm{AD}_1^*(\eta)(z) + 2r^*(\eta)\mathrm{AD}_2^*(\xi)(z)  + \xi(\mathrm{AD}_2^*(\eta)(z))u_r - \eta(\mathrm{AD}_2^*(\xi)(z))u_r. \]
   	That is
   	$(\xi \otimes \eta) \Delta(z) u_r  = - 2r^*(\xi)\mathrm{AD}_2^*(\eta)(z) - 2r^*(\eta)\mathrm{AD}_1^*(\xi)(z) + 2r^*(\xi)\mathrm{AD}_1^*(\eta)(z)+ 2r^*(\eta)\mathrm{AD}_2^*(\xi)(z)$
   	\[
   	+ \xi(\mathrm{AD}_2^*(\eta)(z))u_r - \eta(\mathrm{AD}_2^*(\xi)(z))u_r  +  r^*(\eta)(\xi)(z) u_r - 2r^*(\xi)(\eta)(z) u_r.\]
   	The proof is complete.
   \end{proof}
   \begin{definition}\label{37}
       A pair of multiplicative \(\omega\)-Lie algebras \((L, [\cdot, \cdot]_L, r)\) and \((L^*, [\cdot, \cdot]_{L^*}, r^*)\) with \(\Delta\) given by
       \[
       \langle \Delta(x), \xi \otimes \eta \rangle = \langle [\xi, \eta]_{L^*} + r^*(\eta)\xi - 2r^*(\xi)\eta, x \rangle
       \]
       is called a multiplicative \(\omega\)-Lie bialgebra if the following conditions hold:
       \begin{enumerate}
       	\item $\Delta([x, y]_L) = (\mathrm{id} \otimes \mathrm{ad}_2 x + \mathrm{ad}_2 x \otimes \mathrm{id})\Delta(y) - (\mathrm{id} \otimes \mathrm{ad}_2 y + \mathrm{ad}_2 y \otimes \mathrm{id})\Delta(x) + 2r(y)\Delta(x) - 2r(x)\Delta(y) \\
       	\quad + \bigl(\mathrm{ad}_2(x)(u_r) - 2r(x)u_r\bigr) \otimes y - \bigl(\mathrm{ad}_2(y)(u_r) - 2r(y)u_r\bigr) \otimes x,$
       	
       	\item $(\xi \otimes \eta) \Delta(z) u_r  = - 2r^*(\xi)\mathrm{AD}_2^*(\eta)(z) - 2r^*(\eta)\mathrm{AD}_1^*(\xi)(z) + 2r^*(\xi)\mathrm{AD}_1^*(\eta)(z) + 2r^*(\eta)\mathrm{AD}_2^*(\xi)(z) \\
       	\quad + \xi(\mathrm{AD}_2^*(\eta)(z))u_r - \eta(\mathrm{AD}_2^*(\xi)(z))u_r + r^*(\eta)(\xi)(z) u_r - 2r^*(\xi)(\eta)(z) u_r.$
       \end{enumerate}
   \end{definition}
   
   From Theorems \ref{35} and \ref{36}, we obtain the following conclusion.
   \begin{theorem}\label{38}
       Let \((L, [\cdot, \cdot]_L, r)\) and \((L^*, [\cdot, \cdot]_{L^*}, r^*)\) be two multiplicative \(\omega\)-Lie algebras. Then the following conditions are equivalent:
       \begin{enumerate}
       	\item \((L, L^*)\) is a multiplicative \(\omega\)-Lie bialgebra.
       	\item \((L, L^*; (ad^*_1, ad^*_2), (AD^*_1, AD^*_2))\) is a matched pair of multiplicative \(\omega\)-Lie algebras.
       	\item \((L \oplus L^*; L, L^*)\) is a Manin triple of multiplicative \(\omega\)-Lie algebras associated to the invariant bilinear form given by Eq. (\ref{123}).
       \end{enumerate}
   \end{theorem}
   When multiplicative $\omega$-Lie algebras degenerate to Lie algebras, i.e.,$r=0$, our definitions and equivalence theorems will degenerate to the classical results on Lie algebras.
   
   \section{Yang-Baxter $\omega$-Lie Bialgebras and the $\omega$-Yang-Baxter Equation}\label{4}
   In this section, we introduce the notions of Yang--Baxter $\omega$-Lie bialgebras and the $\omega$-Yang--Baxter equation. We then study their relationship.
   
   \begin{definition}\label{411}
   	   Let $(L, [\cdot, \cdot]_L, r)$ be a multiplicative $\omega$-Lie algebra. Take an element $R = \sum\limits_i x_i \otimes y_i \in L \otimes L$, where the elements $x_i, y_i$ satisfy the following conditions:
   	   \begin{enumerate}
   	   	\item $x_i, y_i \in \ker r$,
   	   	\item $r \circ \mathrm{ad}_1 x_i = r \circ \mathrm{ad}_1 y_i = 0$.
   	   \end{enumerate}
   	   Fix an element $u_r \in C(L)$. Define $[R, R]_L \in L^{\otimes 3}$ by
   	   \[
   	   [R, R]_L := \sum\limits_{i,j} \left( [x_i, x_j]_L \otimes y_i \otimes y_j + x_i \otimes [y_i, x_j]_L \otimes y_j + x_i \otimes x_j \otimes [y_i, y_j]_L \right)
   	   \]
   	   \[
   	   + 3\sum\limits_i \left( y_i \otimes x_i \otimes u_r + x_i \otimes u_r \otimes y_i + u_r \otimes y_i \otimes x_i \right).
   	   \]
   	   The equation $[R, R]_L = 0$ is called the $\omega$-Yang--Baxter equation. 
   \end{definition}
   When the multiplicative $\omega$-Lie algebra degenerates to a Lie algebra, the $\omega$-Yang-Baxter equation reduces to the classical Yang-Baxter equation.
   
   Fix an element $ u_r \in C(L) $. Let $ R = \sum_i x_i \otimes y_i \in L \otimes L $ be an element satisfying the conditions in Definition~\ref{411}. Define the linear map $\Delta : L \to L \otimes L$ by
   \[
   \Delta(x) = \mathrm{ad}_x R = [x, R]_L - 2x \otimes u_r + u_r \otimes x,\tag{4.1}\label{41}
   \]
   where 
   \[[x, R]_L = (\mathrm{ad_1}x \otimes \mathrm{id} + \mathrm{id} \otimes \mathrm{ad_1}x)(R).
   \]
   
   and linear map $\mathrm{Jac}_\Delta(x) \colon L \to L \otimes L$ by
   \[
   \mathrm{Jac}_\Delta(x) = \sum\limits_{c, p} (\mathrm{id} \otimes \Delta)\Delta(x) + 2\sigma(\Delta(x)) \otimes u_r + u_r \otimes \Delta(x) + 2u_r \otimes u_r \otimes x,
   \]
   where ``$\sum\limits_{c,p}$ " means that the sum is taken of the term after the summation sign and together with the two similar terms obtained by cyclic permutations of the factors in the tensor product $\otimes^3 L$, and $\sigma$ is the swap operator acting on the tensor space.
   \begin{lemma}\label{142}
   	   Under the above notation, if $[x, R + \sigma(R)]_L = 0$ for all $x \in L$, then
   	  \[
   	  \mathrm{Jac}_\Delta(x) = \mathrm{ad}_x[R, R]_L := (\mathrm{ad}_1 x \otimes \mathrm{id} \otimes \mathrm{id} + \mathrm{id} \otimes \mathrm{ad}_1 x \otimes \mathrm{id} + \mathrm{id} \otimes \mathrm{id} \otimes \mathrm{ad}_1 x)[R, R]_L.
   	  \]
   \end{lemma}
   \begin{proof}
   	Let $R = \sum\limits_i x_i \otimes y_i$. For any $x \in L$, the condition $[x, R + \sigma(R)]_L = 0$ is equivalent to
   	\[
   	\sum\limits_i \big( [x, x_i]_L \otimes y_i + x_i \otimes [x, y_i]_L \big) = \sum\limits_i \big( -[x, y_i]_L \otimes x_i - y_i \otimes [x, x_i]_L \big)\tag{4.2}\label{422}.
   	\]
   	By explicitly expanding $\mathrm{Jac}_\Delta(x)$ using the expression $R = \sum\limits_i x_i \otimes y_i$, with summation over repeated indices implied, we obtain
   	\[
   	\mathrm{Jac}_\Delta(x)
   	\]
   	\[
   	= \sum\limits_i \sum\limits_j [x, x_i]_L \otimes [y_i, x_j]_L \otimes y_j + \sum\limits_i \sum\limits_j [x, x_i]_L \otimes x_j \otimes [y_i, y_j]_L+ \sum\limits_i \sum\limits_j x_i \otimes [[x, y_i]_L, x_j]_L \otimes y_j
   	\]
   	\[
   	+ \sum\limits_i \sum\limits_j x_i \otimes x_j \otimes [[x, y_i]_L, y_j]_L+ \sum\limits_i \sum\limits_j [y_i, x_j]_L \otimes y_j \otimes [x, x_i]_L + \sum\limits_i \sum\limits_j x_j \otimes [y_i, y_j]_L \otimes [x, x_i]_L
   	\]
   	\[
   	+ \sum\limits_i \sum\limits_j [[x, y_i]_L, x_j]_L \otimes y_j \otimes x_i + \sum\limits_i \sum\limits_j x_j \otimes [[x, y_i]_L, y_j]_L \otimes x_i+ \sum\limits_i \sum\limits_j y_j \otimes [x, x_i]_L \otimes [y_i, x_j]_L 
   	\]
   	\[
   	+ \sum\limits_i \sum\limits_j [y_i, y_j]_L \otimes [x, x_i]_L \otimes x_j+ \sum\limits_i \sum\limits_j y_j \otimes x_i \otimes [[x, y_i]_L, x_j]_L + \sum\limits_i \sum\limits_j [[x, y_i]_L, y_j]_L \otimes x_i \otimes x_j
   	\]
   	\[
   	+ 3\sum\limits_i y_i \otimes [x, x_i]_L \otimes u_r + 3\sum\limits_i [x, y_i]_L \otimes x_i \otimes u_r+ 3\sum\limits_i [x, x_i]_L \otimes u_r \otimes y_i
   	\]
   	\[
   	+ 3\sum\limits_i x_i \otimes u_r \otimes [x, y_i]_L+ 3\sum\limits_i u_r \otimes y_i \otimes [x, x_i]_L + 3\sum\limits_i u_r \otimes [x, y_i]_L \otimes x_i.
   	\]
   	
   	Then we write out $\mathrm{Jac}_\Delta(x) - \mathrm{ad}_x[R, R]_L$ explicitly. Note that $\mathrm{Jac}_\Delta(x)$ is a sum of eighteen terms and $\mathrm{ad}_x[R, R]_L$ is a sum of fifteen terms, but eight terms appear in both expressions and hence are canceled. Thus, $\mathrm{Jac}_\Delta(x) - \mathrm{ad}_x[R, R]_L$ consists of seventeen remaining terms. After suitably rearranging these terms, we obtain
   	\begin{align*}
   		& \mathrm{Jac}_{\Delta}(x) - \mathrm{ad}_x[R, R]_L \\
   		= &\, [y_i, x_j]_L \otimes y_j \otimes [x, x_i]_L - [x_i, x_j]_L \otimes y_i \otimes [x, y_j]_L + [[x, y_i]_L, y_j]_L \otimes x_i \otimes x_j + [y_i, y_j]_L \otimes [x, x_i]_L \otimes x_j \\
   		& - [x_i, x_j]_L \otimes [x, y_i]_L \otimes y_j + [[x, y_i]_L, x_j]_L \otimes y_j \otimes x_i - [x, [x_i, x_j]_L]_L \otimes y_i \otimes y_j + x_i \otimes [[x, y_i]_L, x_j]_L \otimes y_j \\
   		& + x_i \otimes x_j \otimes [[x, y_i]_L, y_j]_L + x_j \otimes [y_i, y_j]_L \otimes [x, x_i]_L + x_j \otimes [[x, y_i]_L, y_j]_L \otimes x_i + y_j \otimes [x, x_i]_L \otimes [y_i, x_j]_L \\
   		& + y_j \otimes x_i \otimes [[x, y_i]_L, x_j]_L - x_i \otimes [x, [y_i, x_j]_L]_L \otimes y_j - x_i \otimes [x, x_j]_L \otimes [y_i, y_j]_L - x_i \otimes [y_i, x_j]_L \otimes [x, y_j]_L \\
   		& - x_i \otimes x_j \otimes [x, [y_i, y_j]_L]_L.
   	\end{align*}
   	
   	Interchanging the indices $i$ and $j$ in the first and second terms and using the equation \ref{422}, their sum becomes
   	\[
   	y_i \otimes y_j \otimes [x, [x_j, x_i]_L]_L + x_i \otimes y_j \otimes [x, [x_j, y_i]_L]_L.
   	\]
   	Using the equation \ref{422}, the sum of the third and fourth terms is
   	\[
   	[y_j, x_i]_L \otimes [x, y_i]_L \otimes x_j + [y_j, [x, x_i]_L]_L \otimes y_i \otimes x_j.
   	\]
   	Similarly, the sum of the term $[y_j, x_i]_L \otimes [x, y_i]_L \otimes x_j$ and the fifth term becomes
   	\[
   	x_i \otimes [x, y_j]_L \otimes [x_j, y_i]_L + y_i \otimes [x, y_j]_L \otimes [x_j, x_i]_L,
   	\]
   	and using the $\omega$-Jacobi identity, the sum of the term $[y_j, [x, x_i]_L]_L \otimes y_i \otimes x_j$ and the sixth term is
   	\[
   	[[x_i, y_j]_L, x]_L \otimes y_i \otimes x_j.
   	\]
   	Furthermore, the sum of $[[x_i, y_j]_L, x]_L \otimes y_i \otimes x_j$ and the seventh term becomes
   	\[
   	- y_j \otimes y_i \otimes [x_i, [x, x_j]_L]_L - x_j \otimes y_i \otimes [x_i, [x, y_j]_L]_L.
   	\]
   	
   	Substituting these results, we find that the expression of $\mathrm{Jac}_\Delta(x) - \mathrm{ad}_x[R, R]_L$ can be written in the form $\sum\limits_i (x_i \otimes u_i + y_i \otimes v_i)$, where $u_i, v_i \in L$. In fact,
   	\begin{align*}
   		u_i ={}& [y_j, y_i] \otimes [x, x_j] - [y_i, x_j] \otimes [x, y_j] + y_j \otimes [x, [x_j, y_i]] - x_j \otimes [x, [y_i, y_j]] \\
   		&+ [x, y_j] \otimes [x_j, y_i] - [x, x_j] \otimes [y_i, y_j] - [x, [y_i, x_j]] \otimes y_j + [[x, y_j], y_i] \otimes x_j \\
   		&+ [[x, y_i], x_j] \otimes y_j + x_j \otimes [[x, y_i], y_j] + y_j \otimes [[x, y_i], x_j].
   	\end{align*}
   	On the right-hand side, the sum of the first four terms is zero by equation \ref{422}, and that of the next three terms becomes
   	\[
   	[x, [y_i, y_j]_L]_L \otimes x_j.
   	\]
   	By the $\omega$-Jacobi identity in $L$, the sum of $[x, [y_i, y_j]_L]_L \otimes x_j$ and the eighth term is
   	\[
   	[[x, y_i]_L, y_j]_L \otimes x_j.
   	\]
   	Furthermore, replacing $x$ in equation \ref{422} by $[x, y_j]$, the sum of $[[x, y_i]_L, y_j]_L \otimes x_j$ and the last three terms becomes
   	\[
   	[[x, y_i]_L, y_j]_L \otimes x_j + [[x, y_i]_L, x_j]_L \otimes y_j + x_j \otimes [[x, y_i]_L, y_j]_L + y_j \otimes [[x, y_i]_L, x_j]_L = 0.
   	\]
   	Thus, we obtain $u_i = 0$. Similarly, one can prove that
   	\[
   	v_i = [x, x_i] \otimes [y_j, x_i] + [x, y_j] \otimes [x_j, x_i] + x_j \otimes [[x, y_j], x_i] + y_j \otimes [[x, x_i], x_j] + y_j \otimes [x, [x_j, x_i]] = 0.
   	\]
   	Hence the conclusion holds.  	   
   \end{proof}
   \begin{definition}\label{43}
   	   A pair of multiplicative \(\omega\)-Lie algebras $(L, [\cdot, \cdot]_L, r)$ and $(L^*, [\cdot, \cdot]_{L^*}, r^*)$ with $\Delta$ given by 
   	   \[
   	   \langle \xi \otimes \eta, \Delta(x) \rangle = \langle [\xi, \eta]_{L^*} + r^*(\eta)\xi - 2r^*(\xi)\eta, x \rangle.
   	   \]
   	   is called a Yang-Baxter $\omega$-Lie bialgebra if 
   	   \[
   	   \Delta[x, y] =  \ ( \mathrm{ad}_2 x \otimes \mathrm{id} + \mathrm{id} \otimes \mathrm{ad}_2 x ) \Delta(y) - ( \mathrm{ad}_2 y \otimes \mathrm{id} + \mathrm{id} \otimes \mathrm{ad}_2 y ) \Delta(x)  \ - 2r([y, x])R + 2r(y)x \otimes u_r
   	   \]
   	   \[
   	   - 2r(x)y \otimes u_r  \ - u_r \otimes [x, y]  \ - u_r \otimes r(y)x + u_r \otimes r(x)y + 2[x, y] \otimes u_r + 3r(u_r)y \otimes x - 3r(u_r)x \otimes y,
   	   \]
   	   where $\left\langle r^*, a \right\rangle = \left\langle a, u_r \right\rangle, \forall a \in L^*$, $R \in L \otimes L$.
   \end{definition}
   \begin{theorem}\label{143}
       Let $(L, [\cdot, \cdot]_L, r)$ be a multiplicative $\omega$-Lie algebra. Fix an element $u_r \in C(L)$. Take $R = \sum\limits_i x_i \otimes y_i \in L \otimes L$ satisfying the conditions in Definition \ref{411}. Define a bracket $[\cdot, \cdot]_{L^*} : L^* \otimes L^* \to L^*$ and a linear form $r^* : L^* \to \mathrm{K}$ by
       \[
       \langle [a, b]_{L^*} + b(u_r)a - 2a(u_r)b, x \rangle = \langle a \otimes b, \Delta(x) \rangle,  \forall\, a, b \in L^*,\ x \in L, 
       \]
       \[
       \langle r^*, a \rangle = \langle a, u_r \rangle,  \forall\, a \in L^*. 
       \]
       Define $\Delta$ by equation(\ref{41}). Then $(L^*, [\cdot, \cdot]_{L^*}, r^*)$ is a multiplicative $\omega$-Lie algebra if and only if the following two conditions hold:
       \begin{enumerate}
       	\item $[x, R + \sigma(R)]_L = 0$ for all $x \in L$,
       	\item $\mathrm{ad}_x [R, R]_L = 0$ for all $x \in L$.
       \end{enumerate}
       Under these conditions, $(L, L^*)$ is a Yang-Baxter $\omega$-Lie bialgebra.
   \end{theorem}
   \begin{proof}
   	$(L^*, [\cdot,\cdot]_{L^*}, r^*)$ is a multiplicative $\omega$-Lie algebra if and only if the bracket $[\cdot,\cdot]_{L^*}$ is skew-symmetric and satisfies the $\omega$-Jacobi identity. The proof that $R$ satisfies the condition~\textup{(I)} if and only if $[\cdot,\cdot]_{L^*}$ is skew-symmetric is straightforward.
   	
   	For all $a, b, c \in L^*$ and $x \in L$, we have
   	\begin{align*}
   		&\left\langle [a, [b, c]_{L^*}]_{L^*} + [b, [c, a]_{L^*}]_{L^*} + [c, [a, b]_{L^*}]_{L^*} + r^*([b, c]_{L^*})a + r^*([c, a]_{L^*})b + r^*([a, b]_{L^*})c, x \right\rangle \\
   		=& \langle a \otimes b \otimes c, \mathrm{Jac}_\Delta(x) \rangle.
   	\end{align*}
   	It is clear that $[\cdot,\cdot]_{L^*}$ satisfies the $\omega$-Jacobi identity if and only if $\mathrm{Jac}_\Delta$ is the zero map. The proof that condition (II) is equivalent to $\mathrm{Jac}_\Delta$ being the zero map follows from Lemma \ref{142}.
   	
   	For all $x, y \in L$, from equation (\ref{41}), we have
   	\begin{align*}
   		\Delta([x, y]_L) &= (\mathrm{ad}_1 [x, y] \otimes \mathrm{id} + \mathrm{id} \otimes \mathrm{ad}_1 [x, y]) R - 2 [x, y] \otimes u_r + u_r \otimes [x, y] \\
   		&= (\mathrm{ad}_1 [x, y] \otimes \mathrm{id} + \mathrm{id} \otimes \mathrm{ad}_1 [x, y])(x_i \otimes y_i) - 2 [x, y] \otimes u_r + u_r \otimes [x, y] \\
   		&= [[x, y], x_i] \otimes y_i + x_i \otimes [[x, y], y_i] - 2 [x, y] \otimes u_r + u_r \otimes [x, y].
   	\end{align*}
   	Moreover, we get
   	\begin{align*}
   		&(\mathrm{ad}_2 x \otimes \mathrm{id} + \mathrm{id} \otimes \mathrm{ad}_2 x)\Delta(y) - (\mathrm{ad}_2 y \otimes \mathrm{id} + \mathrm{id} \otimes \mathrm{ad}_2 y)\Delta(x) \\
   		&= (\mathrm{ad}_2 x \otimes \mathrm{id} + \mathrm{id} \otimes \mathrm{ad}_2 x)\left( (\mathrm{ad}_1 y \otimes \mathrm{id} + \mathrm{id} \otimes \mathrm{ad}_1 y) R - 2y \otimes u_r + u_r \otimes y \right) \\
   		&\quad - (\mathrm{ad}_2 y \otimes \mathrm{id} + \mathrm{id} \otimes \mathrm{ad}_2 y)\left( (\mathrm{ad}_1 x \otimes \mathrm{id} + \mathrm{id} \otimes \mathrm{ad}_1 x) R - 2x \otimes u_r + u_r \otimes x \right) \\
   		&= (\mathrm{ad}_2 x \otimes \mathrm{id} + \mathrm{id} \otimes \mathrm{ad}_2 x)(\mathrm{ad}_1 y \otimes \mathrm{id} + \mathrm{id} \otimes \mathrm{ad}_1 y)R - (\mathrm{ad}_2 y \otimes \mathrm{id} + \mathrm{id} \otimes \mathrm{ad}_2 y)(\mathrm{ad}_1 x \otimes \mathrm{id} + \mathrm{id} \otimes \mathrm{ad}_1 x)R \\
   		&\quad - 2(\mathrm{ad}_2 x \otimes \mathrm{id} + \mathrm{id} \otimes \mathrm{ad}_2 x)(y \otimes u_r) + (\mathrm{ad}_2 x \otimes \mathrm{id} + \mathrm{id} \otimes \mathrm{ad}_2 x)(u_r \otimes y) \\
   		&\quad + 2(\mathrm{ad}_2 y \otimes \mathrm{id} + \mathrm{id} \otimes \mathrm{ad}_2 y)(x \otimes u_r) - (\mathrm{ad}_2 y \otimes \mathrm{id} + \mathrm{id} \otimes \mathrm{ad}_2 y)(u_r \otimes x) \\
   		&= (\mathrm{ad}_2 x \otimes \mathrm{id})(\mathrm{ad}_1 y \otimes \mathrm{id}) R - (\mathrm{ad}_2 y \otimes \mathrm{id})(\mathrm{ad}_1 x \otimes \mathrm{id}) R \\
   		&\quad + (\mathrm{id} \otimes \mathrm{ad}_2 x)(\mathrm{id} \otimes \mathrm{ad}_1 y) R - (\mathrm{id} \otimes \mathrm{ad}_2 y)(\mathrm{id} \otimes \mathrm{ad}_1 x) R \\
   		&\quad - 2(\mathrm{ad}_2 x \otimes \mathrm{id} + \mathrm{id} \otimes \mathrm{ad}_2 x)(y \otimes u_r) + (\mathrm{ad}_2 x \otimes \mathrm{id} + \mathrm{id} \otimes \mathrm{ad}_2 x)(u_r \otimes y) \\
   		&\quad + 2(\mathrm{ad}_2 y \otimes \mathrm{id} + \mathrm{id} \otimes \mathrm{ad}_2 y)(x \otimes u_r) - (\mathrm{ad}_2 y \otimes \mathrm{id} + \mathrm{id} \otimes \mathrm{ad}_2 y)(u_r \otimes x) \\
   		&= [x_i, [y, x]] \otimes y_i + r([y, x])x_i \otimes y_i + x_i \otimes [y_i, [y, x]] + r([y, x])x_i \otimes y_i \\
   		&\quad - 2[x, y] \otimes u_r - 2r(y)x \otimes u_r - 2y \otimes r(u_r)x + r(u_r)x \otimes y + u_r \otimes [x, y] + u_r \otimes r(y)x \\
   		&\quad + 2[y, x] \otimes u_r + 2r(x)y \otimes u_r + 2x \otimes r(u_r)y - r(u_r)y \otimes x - u_r \otimes [y, x] - u_r \otimes r(x)y.
   	\end{align*}
   	Therefore, we obtain
   	\begin{align*}
   		\Delta[x, y]_L &= (\mathrm{ad}_2 x \otimes \mathrm{id} + \mathrm{id} \otimes \mathrm{ad}_2 x)\Delta(y) - (\mathrm{ad}_2 y \otimes \mathrm{id} + \mathrm{id} \otimes \mathrm{ad}_2 y)\Delta(x) - 2r([y, x])R \\
   		&\quad + 2r(y)x \otimes u_r - 2r(x)y \otimes u_r - u_r \otimes [x, y] - u_r \otimes r(y)x + u_r \otimes r(x)y + 2[x, y] \otimes u_r \\
   		&\quad + 3r(u_r)y \otimes x - 3r(u_r)x \otimes y.
   	\end{align*}
   	The compatibility conditions for a Yang-Baxter $\omega$-Lie bialgebra in Definition \ref{43} hold. We complete the proof.
   \end{proof}
   \begin{remark}
      If $R \in L \otimes L$ satisfies the following two conditions: $R + \sigma(R) = 0$ and $[R, R]_L = 0$, then conditions \textup{(I)} and \textup{(II)} in Theorem~\ref{143} hold automatically. Therefore, a Yang--Baxter $\omega$-Lie bialgebra can be constructed from a skew-symmetric solution of the $\omega$-Yang--Baxter equation.
   \end{remark}
   We observe that for multiplicative $\omega$-Lie algebras, the structure associated with the $\omega$-Yang-Baxter equation is the Yang-Baxter $\omega$-Lie bialgebra, which does not coincide with the multiplicative $\omega$-Lie bialgebra that is equivalent to Manin triples and matched pairs. This phenomenon differs from the case of Lie algebras, and similar behavior has also been observed in other algebraic settings, such as Leibniz algebras \cite{cite2} and nearly associative algebras \cite{cite3}. This discrepancy reflects an essential non-uniformity among various algebraic structures. However, when the multiplicative $\omega$-Lie algebra degenerates to a Lie algebra, both the Yang-Baxter $\omega$-Lie bialgebra and the multiplicative $\omega$-Lie bialgebra reduce to the classical Lie bialgebra.
   
   \section{$(\omega)$-Left-symmetric algebras, $\omega$-$\mathcal{O}$-operators and the $\omega$-Yang-Baxter equation}\label{5}
   In this section, we introduce the $\omega$-$\mathcal{O}$-operator, and study its relationship with $\omega$-left-symmetric algebras and left-symmetric algebras respectively. We use the $\omega$-$\mathcal{O}$-operator to construct solutions of the $\omega$-Yang-Baxter equation in a semidirect product $\omega$-Lie algebra.
   \begin{definition}
       Let $(L, [\cdot, \cdot]_L, r)$ be a multiplicative $\omega$-Lie algebra. Take $R = \sum\limits_{i=1}^s x_i \otimes y_i \in L \otimes L$, $u_r \in C(L)$. The equation
       \[
       [R_{12}, R_{13}] + [R_{12}, R_{23}] + [R_{13}, R_{23}] = 0.
       \]
       is called the ``tensor form'' of the $\omega$-Yang--Baxter equation, where
       \begin{align*}
       	R_{12} &= \sum\limits_{i = 1}^s x_i \otimes y_i \otimes 1, \quad 
       	R_{13} = \sum\limits_{i = 1}^s x_i \otimes 1 \otimes y_i, \quad 
       	R_{23} = \sum\limits_{i = 1}^s 1 \otimes x_i \otimes y_i,
       \end{align*}
       and for all $a, b, c, d, e, f \in L \cup \{1\}$,
       \[
       [a \otimes b \otimes c, d \otimes e \otimes f] := [a, d] \otimes b \otimes f + a \otimes [b, e] \otimes f + a \otimes e \otimes [c, f] + \frac{3}{2s}\Big([a, u_r] \otimes b \otimes c + a \otimes [b, u_r] \otimes c 
       \]
       \[
       + b \otimes a \otimes [c, u_r] + [d, u_r] \otimes f \otimes e + d \otimes [e, u_r] \otimes f + d \otimes e \otimes [f, u_r]\Big),
       \]
       where
       \[
       [1, u_r] := u_r, \quad [1, h] := 0,  \forall\, h \in L \setminus \{u_r\}, \quad [a, b] := [a, b]_L,  \forall\, a, b \in L.
       \]
   \end{definition}
   \begin{proposition}
   	   The two forms of the $\omega$-Yang--Baxter equation, namely
   	   \[
   	   [R_{12}, R_{13}] + [R_{12}, R_{23}] + [R_{13}, R_{23}] = 0
   	   \quad \text{and} \quad
   	   [R, R]_L = 0,
   	   \]
   	   are equivalent.
   \end{proposition}
   \begin{proof}
   	   The proof is straightforward and omitted.
   \end{proof}
   \begin{definition}
       	Let $(L, [\cdot, \cdot]_L, r)$ be a multiplicative $\omega$-Lie algebra and \( (\rho, V) \) be a representation of $(L, [\cdot, \cdot]_L, r)$. A linear map $T: V \to L$ is called an $\omega$-$\mathcal{O}$-operator associated to \( (\rho, V) \) if
       	\[
       	[T(u), T(v)] = T(\rho(T(u))v - \rho(T(v))u) + 2r(T(v))T(u) - 2r(T(u))T(v),\forall u,v \in V.
       	\]
   \end{definition}
   When the multiplicative $\omega$-Lie algebra degenerates to a Lie algebra, the $\omega$-$\mathcal{O}$-operator associated to $(\rho, V)$ reduces to an $\mathcal{O}$-operator of the Lie algebra.
   \begin{definition}\cite{cite12}
       An $\omega$-left-symmetric algebra is a triple $(V, \cdot, \omega)$ consisting of a vector space $V$ over a field $\mathrm{K}$, a bilinear multiplication $(u, v) \mapsto u \cdot v$, and a bilinear form $\omega: V \times V \to \mathrm{K}$ satisfying
       \[
       (u \cdot v) \cdot w - u \cdot (v \cdot w) - (v \cdot u) \cdot w + v \cdot (u \cdot w) = \omega(u, v)\, w,  \forall\, u, v, w \in V.
       \]
   \end{definition}
   \begin{definition}
   	   Let $(V, \cdot, \omega)$ be an $\omega$-left-symmetric algebra. If there is a linear form $r : V \to \mathrm{K}$ such that
   	   \[
   	   \omega(u, v) = r(u \cdot v) - r(v \cdot u),  \forall\, u, v \in V.
   	   \]
   	   then the triple $(V, \cdot, r)$ is called a multiplicative $\omega$-left-symmetric algebra.
   \end{definition}
   \begin{proposition}\cite{cite12}
   	   Let $(V, \cdot, \omega)$ be an $\omega$-left-symmetric algebra. Define a triple $(g(V), [\cdot, \cdot], \omega)$, where $g(V) = V$ as a vector space, and the bracket $[\cdot, \cdot]$ is given by
   	   \[
   	   [u, v] = u \cdot v - v \cdot u,  \forall\, u, v \in V.
   	   \]
   	   Then $(g(V), [\cdot, \cdot], \omega)$ is an $\omega$-Lie algebra, called the sub-adjacent $\omega$-Lie algebra of $(V, \cdot, \omega)$.
   	   
   	   In the case that $(V, \cdot, r)$ is a multiplicative $\omega$-left-symmetric algebra, the corresponding sub-adjacent $\omega$-Lie algebra $(V, [\cdot, \cdot], r)$ is a multiplicative $\omega$-Lie algebra.
    \end{proposition}
    %\begin{proof}
    %	For any $u, v, w \in V$, we have
    %	\[
    %	[u, v] = u \cdot v - v \cdot u = - (v \cdot u - u \cdot v) = -[v, u].
    %	\]
    %	Moreover,
    %	\begin{align*}
    %		[[u, v], w] &= [(u \cdot v - v \cdot u), w] \\
    %		&= (u \cdot v - v \cdot u) \cdot w - w \cdot (u \cdot v - v \cdot u) \\
    %		&= (u \cdot v) \cdot w - (v \cdot u) \cdot w - w \cdot (u \cdot v) + w \cdot (v \cdot u).
    %	\end{align*}
    	%
    %	Similarly,
    %	\[
    %	[[v, w], u] = (v \cdot w) \cdot u - (w \cdot v) \cdot u - u \cdot (v \cdot w) + u \cdot (w \cdot v),
    %	\]
    %	\[
    %	[[w, u], v] = (w \cdot u) \cdot v - (u \cdot w) \cdot v - v \cdot (w \cdot u) + v \cdot (u \cdot w).
    %	\]
    	%
    %	Adding these three terms, we obtain
    %	\begin{align*}
    %		[[u, v], w] + [[v, w], u] + [[w, u], v]
    %		= &\, (u \cdot v) \cdot w - u \cdot (v \cdot w) - (v \cdot u) \cdot w + v \cdot %(u \cdot w) \\
    %		&+ (v \cdot w) \cdot u - v \cdot (w \cdot u) - (w \cdot v) \cdot u + w \cdot (v %\cdot u) \\
    %		&+ (w \cdot u) \cdot v - w \cdot (u \cdot v) - (u \cdot w) \cdot v + u \cdot (w %\cdot v).
    %	\end{align*}
    	%
    %	By the definition of the $\omega$-left-symmetric algebra, they are equal to
    %	\[
    %	\omega(u, v)\, w + \omega(v, w)\, u + \omega(w, u)\, v.
    %	\]
    %	
    %	If there is a linear form $r : V \to \mathrm{K}$ such that $\omega(u, v) = r(u \cdot %v - v \cdot u) = r([u, v])$, then we get
    %	\[
    %	r([u, v])\, w + r([v, w])\, u + r([w, u])\, v.
    %	\]
    	%
   % 	Thus, the $\omega$-Jacobi identity holds for the bracket $[\cdot, \cdot]$.
   % \end{proof}
   	\begin{theorem}\label{57}
        Let $T: V \to L$ be an $\omega$-$\mathcal{O}$-operator associated with a representation \( (\rho, V) \) of a multiplicative $\omega$-Lie algebra $(L, [\cdot, \cdot], r)$. Define a bilinear multiplication $* : V \times V \to V$ and a bilinear form $\overline{\omega} : V \times V \to \mathrm{K}$ by
        \[
        u * v = \rho(T(u))v - 2r(T(u))v,
        \]
        \[
        \overline{\omega}(u, v) = 2r(T(\rho(T(v))u)) - 2r(T(\rho(T(u))v)) + r([T(u), T(v)]).
        \]
        for all $u, v \in V$. Then $(V, *, \overline{\omega})$ is an $\omega$-left-symmetric algebra.
   \end{theorem}
   \begin{proof}
   	For any $u, v, w \in V$, we have
   	\[
   	(u * v) * w - u * (v * w) - (v * u) * w + v * (u * w)
   	\]
   	\begin{align*}
   		&= \rho(T(\rho(T(u))(v)))(w) - 2r(T(u))\rho(T(v))(w) 
   		- 2r(T(\rho(T(u))(v)))(w) + 4r(T(u))r(T(v))(w) \\
   		&\quad - \rho(T(\rho(T(v))(u)))(w) + 2r(T(v))\rho(T(u))(w) 
   		+ 2r(T(\rho(T(v))(u)))(w) - 4r(T(v))r(T(u))(w) \\
   		&\quad - \rho(T(u))\rho(T(v))(w) + 2r(T(v))\rho(T(u))(w) 
   		+ 2r(T(u))\rho(T(v))(w) - 4r(T(u))r(T(v))(w) \\
   		&\quad + \rho(T(v))\rho(T(u))(w) - 2r(T(u))\rho(T(v))(w) 
   		- 2r(T(v))\rho(T(u))(w) + 4r(T(v))r(T(u))(w) \\
   		&= \rho\left(T(\rho(T(u))(v) - \rho(T(v))(u))\right)(w) - \left(\rho(T(u))\rho(T(v)) - \rho(T(v))\rho(T(u))\right)(w) \\
   		&\quad + 2r\left(T(\rho(T(v))(u))\right)(w) - 2r\left(T(\rho(T(u))(v))\right)(w) + 2r(T(v))\rho(T(u))(w) - 2r(T(u))\rho(T(v))(w) \\
   		&= \rho\left([T(u), T(v)] + 2r(T(u))T(v) - 2r(T(v))T(u)\right)(w) - \left(\rho([T(u), T(v)]) - r([T(u), T(v)])\,\mathrm{id}_V\right)(w) \\
   		&\quad + 2r\left(T(\rho(T(v))(u))\right)(w) - 2r\left(T(\rho(T(u))(v))\right)(w) + 2r(T(v))\rho(T(u))(w) - 2r(T(u))\rho(T(v))(w) \\
   		&= \rho([T(u), T(v)])(w) + 2r(T(u))\rho(T(v))(w) - 2r(T(v))\rho(T(u))(w) - \rho([T(u), T(v)])(w) \\
   		&\quad + r([T(u), T(v)])(w) + 2r\left(T(\rho(T(v))(u))\right)(w) - 2r\left(T(\rho(T(u))(v))\right)(w) \\
   		&\quad + 2r(T(v))\rho(T(u))(w) - 2r(T(u))\rho(T(v))(w) \\
   		&= 2r\left(T(\rho(T(v))(u))\right)(w) - 2r\left(T(\rho(T(u))(v))\right)(w) + r([T(u), T(v)])(w).
   	\end{align*}
   	On the other hand, we have
   	\begin{align*}
   		\overline{\omega}(u, v)(w) &= 2r\left(T(\rho(T(v))(u))\right)(w) - 2r\left(T(\rho(T(u))(v))\right)(w) + r([T(u), T(v)])(w).
   	\end{align*}
   	Therefore, we conclude that $(V, *, \overline{\omega})$ is an $\omega$-left-symmetric algebra.
   \end{proof}
   It is straightforward to obtain the following conclusion.
   \begin{proposition}\label{58}
   	   Let $(V, \cdot, r)$ be an multiplicative $\omega$-left-symmetric algebra. Define maps $l_1: V \to \mathrm{gl}(V)$ and $l_2: V \to \mathrm{gl}(V)$ by
   	   \[
   	   l_1 (u)(v) := u \cdot v - 2r(v)u, \quad l_2 (u)(v) := u \cdot v.
   	   \]
   	   for all $u, v \in V$. Then $(l_1, l_2)$ is a Generalized Representation I of $(g(V), [\cdot, \cdot], r)$.
  \end{proposition}
  \begin{definition}
  	  Let $(L, [\cdot, \cdot]_L, r)$ be a multiplicative $\omega$-Lie algebra and $(\rho_1, \rho_2)$ be a Generalized Representation I of $(L, [\cdot, \cdot]_L, r)$. A linear map $T: V \to L$ is called an $\omega$-$\mathcal{O}$-operator associated with the Generalized Representation I $(\rho_1, \rho_2)$ if
  	  \[
  	  [T(u), T(v)]_L = T\big(\rho_1(T(u))(v) - \rho_1(T(v))(u)\big) + 2r(T(v))T(u) - 2r(T(u))T(v),  \forall\, u, v \in V.
  	  \]
  \end{definition}
  The following conclusion is immediate.
  \begin{proposition}\label{510}
  	  Let $(V, \cdot, r)$ be a multiplicative $\omega$-left-symmetric algebra, and let $(l_1, l_2)$ be the Generalized Representation I defined in Proposition \ref{58}. Then the identity map $\mathrm{id}: V \to V$ is an $\omega$-$\mathcal{O}$-operator associated with $(l_1, l_2)$.
  \end{proposition}
  \begin{remark}
      The $\omega$-$\mathcal{O}$-operator constructed from an $\omega$-left-symmetric algebra is associated with a Generalized Representation I, rather than an $\omega$-$\mathcal{O}$-operator associated with a representation. This is a phenomenon that differs from the case of Lie algebras. However, when the multiplicative $\omega$-Lie algebra degenerates to a Lie algebra, the $\omega$-left-symmetric algebra degenerates to a left-symmetric algebra, and both the $\omega$-$\mathcal{O}$-operator associated with the Generalized Representation I and the one associated with the representation reduce to the  $\mathcal{O}$-operator of the Lie algebra.
  \end{remark}
  An $\omega$-$\mathcal{O}$-operator associated with a representation can be constructed from a left-symmetric algebra.
  \begin{proposition}\label{512}
      Let $(A, \cdot)$ be a left-symmetric algebra, and let $U$ be a complement subspace of $[A, A]$ in $A$, i.e., $A = [A, A] \oplus U$. Fix a nonzero complex number ${c_x}$. Define a bracket $[\cdot, \cdot]_A$ and a linear map $r: A \to \mathbb{C}$ by
      \[
      [x, y]_A = x \cdot y - y \cdot x,  \forall\, x, y \in A,
      \]
      \[
      r([A, A]) = 0, \quad r(u) = {c_x},  \forall\, u \in U.
      \]
      Then $(A, [\cdot, \cdot]_A, r)$ is a multiplicative $\omega$-Lie algebra.
  \end{proposition}
  \begin{proof}
  	  The proof is straightforward and omitted.
  \end{proof}
  We have the following results.
  \begin{proposition}\label{513}
      Let $(A, \cdot)$ be a left-symmetric algebra, define a map $\rho: A \to \mathrm{gl}(A)$ by
      \[
      \rho(x)(y) := x \cdot y + 2r(x)y,  \forall\, x, y \in A.
      \]
      Then $(\rho, A)$ is a representation of the multiplicative $\omega$-Lie algebra $(A, [\cdot, \cdot]_A, r)$ defined in Proposition \ref{512}.
  \end{proposition}
  \begin{proposition}\label{515}
  	  Let $(A, \cdot)$ be a left-symmetric algebra, and let $(\rho, A)$ be the representation defined in Proposition \ref{513} for the multiplicative $\omega$-Lie algebra $(A, [\cdot, \cdot]_A, r)$, where $(A, [\cdot, \cdot]_A, r)$ is defined in Proposition \ref{512}. Then the identity map $\mathrm{id}: A \to A$ is an $\omega$-$\mathcal{O}$-operator associated with the representation $(\rho, A)$.
  \end{proposition}
  
  At the end of this section, we study the relationship between $\omega$-$\mathcal{O}$-operators associated with representations and the $\omega$-Yang-Baxter equation, and construct solutions of the $\omega$-Yang-Baxter equation based on such $\omega$-$\mathcal{O}$-operators.
  \begin{proposition}
  	  Given a multiplicative $\omega$-Lie algebra $(L, [\cdot, \cdot]_L, r)$ and a representation \( (\rho, V) \) of $L$, we can construct a new multiplicative $\omega$-Lie algebra $L \ltimes_{\rho} V = (L \oplus V, [\cdot, \cdot]_{L \ltimes_{\rho} V}, \overline{r})$ as the semidirect product,, where the bracket $[\cdot, \cdot]_{L \ltimes_{\rho} V}$ and the linear map $\overline{r}$ are defined by
  	  \[
  	  [(x, u), (y, v)]_{L \ltimes_{\rho} V} = ([x, y]_L, \rho(x)(v) - \rho(y)(u)),
  	  \]
  	  \[
  	  \overline{r}(x, u) = r(x),
  	  \]
  	  for all $x, y \in L$ and $u, v \in V$.
  \end{proposition}
  \begin{proof}
  	   The verification is a routine calculation and is therefore omitted.
  \end{proof}
  Given a multiplicative $\omega$-Lie algebra $(L, [\cdot, \cdot]_L, r)$ and a representation \( (\rho, V) \) of $L$, we consider its dual representation \( (\rho^*, V^*) \). Therefore, we obtain the semidirect product $\omega$-Lie algebra $L \ltimes_{\rho^*} V^*$.
  \begin{lemma}
  	 Let $L$ and $V$ be two vector spaces, any linear map $T: V \to L$ can be viewed as an element $\overline{T} \in (L \ltimes_{\rho^*} V^*) \otimes (L \ltimes_{\rho^*} V^*)$ via 
  	 \[
  	 \overline{T}(\xi + u, \eta + v) = \langle T(u), \eta \rangle,  \forall\, \xi + u, \eta + v \in L^* \oplus V.
  	 \]
  \end{lemma}
  \begin{theorem}\label{518}
  	  Let $(L, [\cdot, \cdot]_L, r)$ be a multiplicative $\omega$-Lie algebra, and \( (\rho, V) \) be a representation of $L$. Let $T: V \to L$ be a linear map. Then $R = \overline{T} - \sigma(\overline{T})$ is a skew-symmetric solution of the $\omega$-Yang-Baxter equation on the multiplicative $\omega$-Lie algebra $L \ltimes_{\rho^*} V^*$ if and only if $T$ is an $\omega$-$\mathcal{O}$-operator associated with the representation \( (\rho, V) \).
  \end{theorem}
  \begin{proof}
  	Let $\{e_1, \dots, e_n\}$ be a basis of $L$. Let $\{v_1, \dots, v_m\}$ be a basis of $V$ and $\{v_1^*, \dots, v_m^*\}$ be its dual basis. We define the element $\overline{T} \in (L \ltimes_{\rho^*} V^*) \otimes (L \ltimes_{\rho^*} V^*)$ by
  	$\overline{T} = \sum\limits_{i=1}^m T(v_i) \otimes v_i^*$. Then we have
  	\[
  	[R_{12}, R_{13}] + [R_{12}, R_{23}] + [R_{13}, R_{23}]
  	\]
  	\[
  	= \sum_{i,k=1}^m [T(v_i), T(v_k)] \otimes v_i^* \otimes v_k^* 
  	- \rho^*(T(v_i)) v_k^* \otimes v_i^* \otimes T(v_k) + \rho^*(T(v_k)) v_i^* \otimes T(v_i) \otimes v_k^* 
  	\]
  	\[
  	- v_i^* \otimes [T(v_i), T(v_k)] \otimes v_k^* - T(v_i) \otimes \rho^*(T(v_k)) v_i^* \otimes v_k^* + v_i^* \otimes \rho^*(T(v_i)) v_k^* \otimes T(v_k)
  	\]
  	\[
  	+ v_i^* \otimes v_k^* \otimes [T(v_i), T(v_k)] + T(v_i) \otimes v_k^* \otimes \rho^*(T(v_k)) v_i^* - v_i^* \otimes T(v_k) \otimes \rho^*(T(v_i)) v_k^*.
  	\]
  	Set
  	\[
  	\mathrm{OT}(u,v) = [T(u), T(v)] + T(\rho(T(v)) u) - T(\rho(T(u)) v) - 2r(T(v)) T(u) + 2r(T(u)) T(v), \forall\, u,v \in V.
  	\]
  	Note that
  	\[
  	- \sum_{i,k} \rho^*(T(v_i)) v_k^* \otimes v_i^* \otimes T(v_k) - \sum_{i,k,j} \langle \rho^*(T(v_i)) v_k^*, v_j \rangle v_j^* \otimes v_i^* \otimes T(v_k)
  	\]
  	\[
  	= - \sum_{i,j,k} \left( -\langle v_k^*, \rho(T(v_i)) v_j \rangle + 2r(T(v_i)) \langle v_k^*, v_j \rangle \right) v_j^* \otimes v_i^* \otimes T(v_k) = \sum_{i,j,k} \langle v_k^*, \rho(T(v_i)) v_j \rangle v_j^* \otimes v_i^* \otimes T(v_k) 
  	\]
  	\[
  	- \sum_{i,j,k} 2r(T(v_i)) \delta_{jk} v_j^* \otimes v_i^* \otimes T(v_k) = \sum_{i,j,k} \langle v_k^*, \rho(T(v_i)) v_j \rangle v_j^* \otimes v_i^* \otimes T(v_k) - \sum_{i,j} 2r(T(v_i)) v_j^* \otimes v_i^* \otimes T(v_j)
  	\]
  	\[
  	= \sum_{i,j,k} \langle v_j^*, \rho(T(v_i)) v_k \rangle v_k^* \otimes v_i^* \otimes T(v_j) - \sum_{i,k} 2r(T(v_i)) v_k^* \otimes v_i^* \otimes T(v_k) = \sum_{i,k} v_k^* \otimes v_i^* \otimes T(\rho(T(v_i)) v_k)
  	\]
  	\[
  	- 2r(T(v_i)) v_k^* \otimes v_i^* \otimes T(v_k).
  	\]
  	Then we get
  	\[
  	[R_{12}, R_{13}] + [R_{12}, R_{23}] + [R_{13}, R_{23}] 
  	\]
  	\[
  	= \sum_{i,k=1}^m \Big( \mathrm{OT}(v_i, v_k) \otimes v_i^* \otimes v_k^* 
  	- v_i^* \otimes \mathrm{OT}(v_i, v_k) \otimes v_k^* 
  	+ v_i^* \otimes v_k^* \otimes \mathrm{OT}(v_i, v_k) \Big).
  	\]
  	Therefore $R = \overline{T} - \sigma(\overline{T})$ is a skew-symmetric solution of the $\omega$-Yang-Baxter equation on the multiplicative $\omega$-Lie algebra $L \ltimes_{\rho^*} V^*$ if and only if $T$ is an $\omega$-$\mathcal{O}$-operator associated with the representation \( (\rho, V) \). 	  
  \end{proof}


\begin{thebibliography}{9999}
   	\bibitem{cite1} A. R. Aghdam, L. S. Ghadim, and G. Haghighatdoost, {\em Leibniz bialgebras, classical Yang-Baxter equations and dynamical systems}, Adv. Appl. Clifford Algebr., 31 (2021), no. 5, Paper No. 77, 20 pp.
   	
   	\bibitem{cite21} M.~Aguiar, {\em Infinitesimal Hopf algebras}, New Trends in Hopf Algebra Theory (La Falda, 1999), 1–29, Contemp. Math., 267, Amer. Math. Soc., Providence, RI, 2000.
   	
   	\bibitem{cite35} M.~Aguiar, {\em On the associative analog of Lie bialgebras}, J. Algebra, \textbf{244} (2001), no.~2, 492–532.
   	
   	\bibitem{cite4} C. Bai, {\em A unified algebraic approach to the classical Yang-Baxter equation,} J. Phys. A, 40 (2007), no. 36, 11073–11082.

   	
   	\bibitem{cite2} E. Barreiro, S. Benayadi, {\em A new approach to Leibniz bialgebras,} Algebr. Represent. Theory, 19 (2016), no. 1, 71–101.
   	
   	\bibitem{cite3} E. Barreiro, S. Benayadi, and C. Rizzo, {\em Bialgebra theory for nearly associative algebras and LR-algebras: equivalence, characterization, and LR-Yang-Baxter equation,} Linear Algebra Appl., 711 (2025), 84–125.
   	
   	
   	
   	\bibitem{cite5} L. Cai and Y. Sheng, {\em Purely Hom-Lie bialgebras,} Sci. China Math., 61 (2018), no. 9, 1553–1566.
   	
   	\bibitem{cite13} Y. Chen, C. Liu, and R. Zhang, {\em Classification of three-dimensional complex $\omega$-Lie algebras,} Port. Math., 71 (2014), no. 2, 97–108.
   	
   	\bibitem{cite15} Y. Chen and R. Zhang, {\em Simple $\omega$-Lie algebras and 4-dimensional $\omega$-Lie algebras over $\mathbb{C}$,} Bull. Malays. Math. Sci. Soc., 40 (2017), no. 3, 1377–1390.
   	
   	\bibitem{cite14} Y. Chen, Z. Zhang, R. Zhang, and R. Zhuang, {\em Derivations, automorphisms, and representations of complex $\omega$-Lie algebras,} Comm. Algebra, 46 (2018), no. 2, 708–726.
   	
   	\bibitem{cite16} Y.~Chen, S.~Ren, J.~Shan, and R.~Zhang, {\em Generalized Derivations of $\omega$-Lie algebras}, 2025. DOI:~\url{10.1142/S0219498826502063}, to appear in JAA.
   	
   	\bibitem{cite10} Z. Chen, J. Ni, and J. Yu, {\em Description of $\omega$-Lie algebras,} J. Geom. Phys., 192 (2023), Paper No. 104926, 13 pp.
   	
   	\bibitem{cite12} Z. Chen and Y. Wu, {\em The classification of $\omega$-left-symmetric algebras in low dimensions,} Bull. Korean Math. Soc., 60 (2023), no. 3, 747–762.
   	
   	\bibitem{cite11} Z.~Chen, J.~Ni, and J.~Yu, {\em The Structure of $\Omega$-Left-Symmetric Algebras}, Available at SSRN: \url{https://ssrn.com/abstract=5260612}.
   	
   	\bibitem{cite18} M. L. Dassoundo and S. Silvestrov, {\em Nearly associative and nearly Hom-associative algebras and bialgebras,}in: Non-commutative and Non-associative Algebra and Analysis Structures, Springer Proc. Math. Stat., 426, 259–284.
   	
   	\bibitem{cite9} V.~G.~Drinfel'd, {\em Constant quasiclassical solutions of the Yang--Baxter quantum equation}, Dokl. Akad. Nauk SSSR, \textbf{273} (1983), no.~3, 531--535.
   	
   	\bibitem{cite17} V.~G.~Drinfel'd, {\em Quantum groups}, in: Proceedings of the International Congress of Mathematicians, Vol.~1, 2 (Berkeley, Calif., 1986), 798–820, Amer. Math. Soc., Providence, RI, 1987.
   	
   	\bibitem{cite8} P.~Etingof and A.~Varchenko, {\em Solutions of the quantum dynamical Yang--Baxter equation and dynamical quantum groups}, Comm. Math. Phys., \textbf{196} (1998), no.~3, 591–640.
   	
   	\bibitem{cite22} B. A. Kupershmidt, {\em What a classical r-matrix really is,} J. Nonlinear Math. Phys., 6 (1999), 448–488.
   	
   	\bibitem{cite23} P. Nurowski, {\em Deforming a Lie algebra by means of a 2-form,} J. Geom. Phys., 57 (2007), 1325–1329.
   	
   	\bibitem{cite24} H. Oubba, {\em Local (2-local) derivations and automorphisms and biderivations of complex $\omega$-Lie algebras,} Matematiche (Catania), 79 (2024), no. 1, 135–150.
   	
   	\bibitem{cite25} H. Oubba, {\em Biderivations, local and 2-local derivation and automorphism of simple $\omega$-Lie algebras,} arXiv preprint, arXiv:2505.00436 (2025).
   	
   	\bibitem{cite27} M. A. Semenov-Tian-Shansky, {\em What is a classical R-matrix?,} Funct. Anal. Appl., 17 (1983), 259–272.
   	
   	\bibitem{cite6} Y. Sheng and C. Bai, {\em A new approach to hom-Lie bialgebras,} J. Algebra, 399 (2014), 232–250.
   	
   	\bibitem{cite19} L.~A.~Takhtadzhyan and L.~D.~Faddeev, {\em The quantum method for the inverse problem and the XYZ Heisenberg model}, Russian Math. Surveys, \textbf{34} (1979), no.~5, 13–63.
   	
   	\bibitem{cite28} R. Tang and Y. Sheng, {\em Leibniz bialgebras, relative Rota-Baxter operators, and the classical Leibniz Yang-Baxter equation,} J. Noncommut. Geom., 16 (2022), no. 4, 1179–1211.
   	
   	\bibitem{cite7}Y. Tao, C. Bai, and L. Guo, {\em Another approach to Hom-Lie bialgebras via Manin triples,} Comm. Algebra, 48 (2020), no. 7, 3109–3132.
   	
   	\bibitem{cite29} D. Yau, {\em The classical Hom-Yang-Baxter equation and Hom-Lie bialgebras,} Int. Electron. J. Algebra, 17 (2015), 11–45.
   	
   	\bibitem{cite30} R. Zhang, {\em Representations of $\omega$-Lie algebras and tailed derivations of Lie algebras,} Internat. J. Algebra Comput., 31 (2021), no. 2, 325–339.
   	
   	\bibitem{cite32} J. Zhou, L. Chen, Y. Ma, and B. Sun, {\em On $\omega$-Lie superalgebras,} J. Algebra Appl., 17 (2018), no. 11, 1850212, 17 pp.
   	
   	\bibitem{cite31} J. Zhou and L. Chen, {\em On low-dimensional complex $\omega$-Lie superalgebras,} Adv. Appl. Clifford Algebr., 31 (2021), no. 3, Paper No. 54, 30 pp.
   	
   	
   	
   	\bibitem{cite33} J. Zhou, {\em Generalized derivations of complex $\omega$-Lie superalgebras,} arXiv preprint, arXiv:2505.22966 (2025).
   	
   	\bibitem{cite34} P. Zusmanovich, {\em $\omega$-Lie algebras,} J. Geom. Phys., 60 (2010), 1028–1044.
   	
   	
   	
   	
   	
   	
   	
   	
   	
   	
   	
   	
   	
   	
   	
   	
   	
   	
   	%\bibitem{cite20} L. D. Faddeev and L. Takhtajan, {\em Hamiltonian Methods in the Theory of Solitons,} Springer, Berlin, 1987.
   	
   	
   	
   	
   	
   
   	
   	
   	
   	
   	
   	%\bibitem{cite26} J. A. Schouten, {\em Uber differentialkomitanten zweier kontravarianter Gr\"ossen,} Ned. Akad. Wet. Proc. A 43 (1940), 449–452.
   	
   	
   	
   	
   	
   	
   	
   	
   	
   	
   	
   	
   \end{thebibliography}
\end{document}